\documentclass{article}%
\usepackage{amsmath}
\usepackage{amsfonts}
\usepackage{amssymb}
\usepackage{graphicx}
\usepackage{color}
\usepackage{hyperref}%
\usepackage[shortlabels]{enumitem}
\usepackage{rotating}
\providecommand{\U}[1]{\protect\rule{.1in}{.1in}}
\newtheorem{theorem}{Theorem}[section]

\newtheorem{assumption}[theorem]{Assumption}

\newtheorem{corollary}[theorem]{Corollary}

\newtheorem{lemma}[theorem]{Lemma}

\newtheorem{remark}[theorem]{Remark}

\newenvironment{proof}[1][Proof]{\noindent\textbf{#1.} }{\ \rule{0.5em}{0.5em}}

\newcommand{\im}{\operatorname*{i}}

\newcommand{\Gw}[1]{\left\| #1 \right\|_{\Gamma,|\zeta|}}

\numberwithin{equation}{section}

\begin{document}

\title{Wave number-Explicit Analysis for Galerkin Discretizations of Lossy Helmholtz Problems}
\author{Jens M.~Melenk\thanks{(melenk@tuwien.ac.at), Institut f\"{u}r Analysis und
Scientific Computing, Technische Universit\"{a}t Wien, Wiedner Hauptstrasse
8-10, A-1040 Wien, Austria.}
\and Stefan A.~Sauter\thanks{(stas@math.uzh.ch), Institut f\"{u}r Mathematik,
Universit\"{a}t Z\"{u}rich, Winterthurerstr 190, CH-8057 Z\"{u}rich,
Switzerland}
\and C\'eline Torres \thanks{(celine.torres@math.uzh.ch), Institut f\"{u}r
Mathematik, Universit\"{a}t Z\"{u}rich, Winterthurerstr 190, CH-8057
Z\"{u}rich, Switzerland}}
\maketitle

\begin{abstract}
We present a stability and convergence theory for the lossy Helmholtz
equation and its Galerkin discretization. The boundary conditions are of Robin
type. All estimates are explicit with respect to the real and imaginary part
of the complex wave number $\zeta\in\mathbb{C}$, $\operatorname{Re}\zeta\geq0$,
$\left\vert \zeta\right\vert \geq1$. For the extreme cases $\zeta
\in\operatorname*{i}\mathbb{R}$ and $\zeta\in\mathbb{R}_{\geq0}$, the
estimates coincide with the existing estimates in the literature and 
exhibit a seamless
transition between these cases in the right complex half plane.

\end{abstract}

\section{Introduction}

For many problems in time-harmonic acoustic scattering, the Helmholtz equation
serves as a model problem, and its numerical discretization is a topic of vivid
research. For homogeneous, isotropic material the differential operator is
given by 
$$
\mathcal{L}_{\zeta}u:=-\Delta u+\zeta^{2}u,
$$
where 
$\zeta= \operatorname{Re} \zeta + \operatorname{i}\operatorname{Im} \zeta =: \nu - k \operatorname{i} $ 
with $\nu > 0$ and $k \in {\mathbb R}$ 
denotes
the \textit{wave number}. The solution is highly oscillatory if
$\left\vert \operatorname{Im}\zeta\right\vert \gg1$, which makes the
discretization challenging with respect to both, stability and accuracy. To
study this problem systematically the case of purely imaginary wave numbers
$\zeta=-\operatorname*{i}k$, $k\in\mathbb{R}$, has often been used in the
literature as a model problem for designing and analyzing numerical
methods. However, in many applications waves are damped, e.g., by
friction and viscoelastic effects in the material or loss via sound radiation
or flow of vibration energy out of the physical scatterer (see, e.g.,
\cite{vorlander2007auralization}).

Another important application is the approximation of the inverse Laplace
transform by contour quadrature where the Helmholtz operator has to be
discretized at many complex frequencies in the right complex half plane (see,
e.g., \cite{lopezsauter_contour}).

For the two extreme cases $\zeta=-\operatorname*{i}k$ and $\zeta=\nu$,
$k\in\mathbb{R}$, $\nu\in\mathbb{R}_{\geq0}$, a fairly complete theory for
standard Galerkin $hp$-finite elements is available
and the
error estimates are explicit with respect to the wave number $\zeta$, the mesh
width $h$ of the finite element mesh, and the polynomial degree $p$: 
a) For
$\zeta=-\operatorname*{i}k$ and large $\left\vert k\right\vert $ the problem
is highly indefinite and a \textquotedblleft resolution
condition\textquotedblright\ of the form%
\[
\frac{|k|h}{p}\leq C\quad\text{ together with }\quad p\geq C\log\left\vert
k\right\vert
\]
has to be imposed in order to ensure solvability of the Galerkin equations and
quasi-optimality 
(\cite{MelenkSauterMathComp,mm_stas_helm2,MPS13,esterhazy-melenk12}); 
b) for $\zeta=\nu>0$ and $\nu=O\left(  1\right)  $, the
problem is properly elliptic and C\'{e}a's lemma ensures well-posedness and
quasi-optimality without any resolution condition; c) for $\zeta=\nu\gg1$, the
solution exhibits boundary layers. Although the Galerkin discretization is
always well-posed in this last situation, special meshes should be used that are
adapted to the boundary layers 
(see, e.g., \cite{miller1996fitted, Shi91, MelenkHabil} and references there). 
In this paper, we will develop a unified theory for Galerkin discretizations
of $\mathcal{L}_{\zeta}$ with Robin boundary conditions that is applicable
for all $\zeta\in\mathbb{C}$, $\operatorname{Re}\zeta\geq0$, and $\left\vert
\zeta\right\vert \geq1$. All estimates are explicit in terms of
$\operatorname{Re}\zeta$ and $\operatorname{Im}\zeta$ and reproduce the limiting 
cases of purely real and imaginary $\zeta$. 
It is shown that, for the \textit{sectorial case}, i.e., the wave number
lies in a sectorial neighborhood of the real axis in the right complex half
plane, well-posedness and quasi-optimality is a consequence of coercivity
while for $\operatorname{Re}\zeta\rightarrow0$ the estimates tend continuously
to the purely imaginary case $\zeta=-\operatorname*{i}k$. We follow the
general theory developed in \cite{MelenkSauterMathComp, mm_stas_helm2}
and refine the estimates to be explicit with respect to the real and imaginary
part of the wave number.

The paper is structured as follows. In Sect.~\ref{SectSet} we introduce the
Helmholtz model problem with Robin boundary conditions and formulate some
geometric and algebraic assumptions on the data. Further, we define for the
wave number the (well-behaved) sectorial and the (more critical) non-sectorial region.

The estimate of the continuity constant for the sesquilinear form is derived
in Sect.~\ref{SectContConst}. Sect.~\ref{SectInfSup} is devoted to the
analysis of the inf-sup constant for the continuous sesquilinear form. If the
real part of the wave number is positive the estimate follows simply from the
coercivity of the sesquilinear form. However, this bound degenerates as 
$\operatorname{Re}\zeta\rightarrow0$. This can be remedied by a different proof:
first one uses suitable test functions to derive stability estimates for an
adjoint problem with $L^{2}$ right-hand sides and then by employing this
result for the estimate of the inf-sup constant in a vicinity of the imaginary axis.

The key role for the analysis of the Galerkin discretization is played by a
regular decomposition of the Helmholtz solution. In Sect.~\ref{SecRegDecomp},
we introduce a splitting of the Helmholtz solution into a part with (low)
$H^{2}$-regularity and wave number-\textit{independent} regularity constant and
an analytic part with a more critical wave number dependence. First, this is
derived for the full space solution by generalizing the results for purely
imaginary frequencies in \cite{MelenkSauterMathComp}. In the case of bounded
domains, we generalize the \textit{iteration argument} in 
\cite[Sect.~4]{mm_stas_helm2} to general complex frequencies. In addition, this requires
sharp estimates of frequency-depending lifting operators which we also present
in this section.

Sect.~\ref{sec:discretization} is devoted to the estimate of the discrete
inf-sup constant for the standard Galerkin discretization of the Helmholtz
equation. We will derive two type of estimates: one requires that the finite
dimensional space for the Galerkin discretization satisfies a certain
\textit{resolution condition} and allows for robust (as $\operatorname{Re}%
\zeta\rightarrow0)$ stability and quasi-optimal convergence estimates; the
other one avoids a resolution condition while the constants in the estimates
tend towards $\infty$ as $\operatorname{Re}\zeta\rightarrow0$ but stay robust
for the \textit{sectorial case}. Numerical examples in Sect.~\ref{sec:numerics}
illustrate the application of our analysis in the context 
of $hp$-FEM. 

\section{Setting\label{SectSet}}

We consider the Helmholtz problem
\begin{equation}%
\begin{split}
-\Delta u+\zeta^{2}u  &  =f\qquad\text{ in }\Omega,\\
\partial_{n}u+\zeta u  &  =g\qquad\text{ on }\Gamma:=\partial\Omega,
\end{split}
\label{eq:strong-helmholtz-robin}%
\end{equation}
for $f\in L^{2}(\Omega)$ and $g\in L^{2}(\Gamma)$. We assume that the
wave number (frequency) $\zeta$ satisfies\footnote{The condition $|\zeta|\geq1$
can be replaced by $\left\vert \zeta\right\vert \geq\rho_{0}$ for any
$\rho_{0}>0$. However, the constants in our estimates, possibly, deteriorate
as $\rho_{0}\rightarrow0$.}%
\begin{equation}
\zeta\in\mathbb{C}_{\geq0}^{\circ}:=\left\{  \zeta\in\mathbb{C}_{\geq0}%
\mid|\zeta|\geq1\right\}  ,
\end{equation}
where, for $\rho\in\mathbb{R}$,
\[
\mathbb{C}_{>\rho}:=\left\{  \xi\in\mathbb{C}\mid\operatorname{Re}\xi
>\rho\right\}  \quad\text{and\quad}\mathbb{C}_{\geq\rho}:=\left\{  \xi
\in\mathbb{C}\mid\operatorname{Re}\xi\geq\rho\right\}  .
\]
Note that the choice $\zeta=-\operatorname*{i}k$ leads to the standard
Helmholtz case. The frequency domain $\mathbb{C}_{\geq0}^{\circ}$ is split
into the \textit{sectorial} and \textit{non-sectorial} cases%
\begin{align*}
S_{\beta}  &  :=\{\xi\in\mathbb{C}_{\geq0}^{\circ}:\left\vert
\operatorname{Im}\xi\right\vert <\beta\operatorname{Re}\xi\},\\
S_{\beta}^{c}  &  :=\{\xi\in\mathbb{C}_{\geq0}^{\circ}:\left\vert
\operatorname{Im}\xi\right\vert \geq\beta\operatorname{Re}\xi\}
\end{align*}
for some $\beta>0$. Our focus is on the derivation of stability and error
estimates that are explicit in the real and imaginary part of $\zeta$ but
less on the development of a theory with minimal assumptions on the geometry
of the domain. In this light we impose the following simplifying assumption.

\begin{assumption}
\label{ASmoothDomain}$\Omega\subset\mathbb{R}^{3}$ is a bounded domain with
analytic boundary that is star-shaped with respect to a ball.
\end{assumption}

We note that our results can be extended to convex polygonal domains in a
straightforward way following the arguments in \cite{mm_stas_helm2}.

Let $L^{2}\left(  \Omega\right)  $ denote the usual Lebesgue space with scalar
product denoted by $\left(  \cdot,\cdot\right)  $ (complex conjugation is on
the second argument) and norm $\|\cdot\|_{L^2(\Omega)}:= \left\Vert \cdot\right\Vert :=\left(
\cdot,\cdot\right)  ^{1/2}$. Let $V=H^{1}\left(  \Omega\right)  $ denote the
usual Sobolev space and let $\gamma_{0}:H^{1}\left(  \Omega\right)
\rightarrow H^{1/2}\left(  \Gamma\right)  $ be the standard trace operator. We
introduce the sesquilinear forms
\[
a_{0,\zeta}\left(  u,v\right)  :=\left(  \nabla u,\nabla v\right)  +\left(
\zeta^{2}u,v\right)  \qquad\forall u,v\in V,
\]
and
\[
b_{\zeta}\left(  \gamma_{0}u,\gamma_{0}v\right)  :=\left(  \zeta\gamma
_{0}u,\gamma_{0}v\right)  _{\Gamma}\qquad\forall u,v\in V,
\]
where $\left(  \cdot,\cdot\right)  _{\Gamma}$ is the $L^{2}\left(
\Gamma\right)  $ scalar product.

The weak formulation of the Helmholtz problem with Robin boundary conditions
\eqref{eq:strong-helmholtz-robin} is given as follows: For $F=(f,\cdot
)+(g,\gamma_0 \cdot)_{\Gamma}\in V^{\prime}$, we seek $u\in V$ such that
\begin{equation}
a_{\zeta}\left(  u,v\right)  :=a_{0,\zeta}\left(  u,v\right)  +b_{\zeta
}\left(  \gamma_{0}u,\gamma_{0}v\right)  =F\left(  v\right)  \quad\forall v\in
V. \label{varformrobin}%
\end{equation}
In the following, we will omit explicitly writing the trace operator
$\gamma_0$ when it is clear that it is implied.

\section{The Continuity Constant\label{SectContConst}}

In this section, we will estimate the continuity constant of the sesquilinear
form $a_{\zeta}\left(  \cdot,\cdot\right)  $. 
We equip the Sobolev space $V$
with the indexed norm $\left\Vert \cdot\right\Vert _{|\zeta|}$, where, for
$\rho>0$, we set%
\begin{equation}
\label{eq:norm-rho}
\left\Vert u\right\Vert _{\rho,\Omega}=\left\Vert u\right\Vert _{\rho
}:=\left(  \left\Vert \nabla u\right\Vert ^{2}+\rho^{2}\left\Vert u\right\Vert
^{2}\right)  ^{1/2}.
\end{equation}
More generally, for measurable subsets $T \subset \Omega$ we write 
\begin{equation*}
\|u\|_{\rho,T} := 
\left( \|\nabla u\|^2_{L^2(T)} + \rho^2 \|u\|^2_{L^2(T)}\right)^{1/2}
\end{equation*}
The $L^{2}$-norm on $\Gamma$ is denoted by $\Vert\cdot\Vert_{\Gamma}$. 
On $H^{1/2}(\Gamma)$ we introduce the weighted norm 
\begin{equation}
\label{eq:norm-rho-bdy}
\left\| g \right\|_{\Gamma,\rho} := \left(\|g\|^2_{H^{1/2}(\Gamma)}+ 
\rho\|g\|^2_{\Gamma}\right)^{1/2},
\end{equation}
for \(\rho>0\).
\begin{theorem}
The sesquilinear form $a_{\zeta}$ is continuous and 
\begin{equation}
\left\vert a_{\zeta}\left(  u,v\right)  \right\vert \leq\left(  1+C_{b}%
\right)  \left\Vert u\right\Vert _{|\zeta|}\left\Vert v\right\Vert _{|\zeta
|}\quad\forall u,v\in H^{1}\left(  \Omega\right)  \label{eq:contRobin}%
\end{equation}
with $C_{b}$ independent of $\zeta\in\mathbb{C}_{\geq0}$.
\end{theorem}

%

\proof
The continuity estimate for the sesquilinear form $b_{\zeta}\left(
\cdot,\cdot\right)  $ is a simple consequence of the multiplicative trace
inequality (see \cite[p.41, last formula]{Grisvard85})%
\begin{equation}
\left\Vert \gamma_{0}u\right\Vert _{\Gamma}\leq C_{\operatorname*{trace}%
}\left\Vert u\right\Vert ^{1/2}\left\Vert u\right\Vert _{H^{1}(\Omega)}^{1/2}.
\label{eq:mtrace}%
\end{equation}
Hence%
\begin{equation}
\sqrt{|\zeta|}\left\Vert \gamma_{0}u\right\Vert _{L^{2}\left(  \Gamma\right)
}\leq C_{\mathrm{trace}}\left(  \left\vert \zeta\right\vert \left\Vert
u\right\Vert \right)  ^{1/2}\left\Vert u\right\Vert _{H^{1}(\Omega)}^{1/2}\leq
C\left\Vert u\right\Vert _{\left\vert \zeta\right\vert }, 
\label{multtraceinequ}%
\end{equation}
which implies the continuity of $b_{\zeta}\left(  \cdot,\cdot\right)  $
\begin{equation}
\left\vert b_{\zeta}\left(  \gamma_{0}u,\gamma_{0}v\right)  \right\vert \leq
C_{b}\left\Vert u\right\Vert _{\left\vert \zeta\right\vert }\left\Vert
v\right\Vert _{\left\vert \zeta\right\vert }\qquad\forall u,v\in H^{1}\left(
\Omega\right)  \label{defCrobin}%
\end{equation}
for a constant $C_{b}$ independent of $\zeta\in\mathbb{C}_{\geq0}^{\circ}$ and
$u$, $v$.
\endproof

\section{The Inf-Sup Constant of $a_{\zeta}\left(  \cdot,\cdot\right)$}\label{SectInfSup}

Our goal in this section is to estimate the inf-sup \textit{constant}
\begin{equation}
\label{eq:inf-sup}
\gamma_{\zeta}:=\inf_{u\in V}\sup_{v\in V}\frac{\left\vert a_{\zeta}\left(
u,v\right)  \right\vert }{\left\Vert u\right\Vert _{\left\vert \zeta
\right\vert }\left\Vert v\right\Vert _{|\zeta|}},%
\end{equation}
which implies well-posedness of (\ref{varformrobin}). This involves two
different theoretical techniques: In Sect.~\ref{subsec:infsupsectorial} we
consider the case $\operatorname{Re}\zeta>0$ and obtain estimates from the
coercivity of the sesquilinear form. These estimates give stable bounds for the
sectorial case but deteriorate as $\operatorname{Re}\zeta\rightarrow0$ in the
non-sectorial case. In Sect.~\ref{subsec:infsuphelmh} we employ the
sesquilinear form with a suitably selected test function and obtain sharp
estimates also for the non-sectorial case.

\subsection{The Inf-Sup Constant for $\operatorname{Re}\zeta>0$}

\label{subsec:infsupsectorial}

The estimate of the inf-sup constant in the following Lemma~\ref{Lem:infsup1} 
is a direct consequence of the technique used in \cite{bambduong}.

\begin{lemma}
\label{Lem:infsup1}Let $\Omega\subset\mathbb{R}^{3}$ be a bounded Lipschitz
domain and let $\zeta\in\mathbb{C}_{>0}^{\circ}$. Then the inf-sup constant 
$\gamma_{\zeta}$ of (\ref{eq:inf-sup}) 
for the sesquilinear form $a_{\zeta}\left(  \cdot,\cdot\right)  $ (cf.\ 
(\ref{varformrobin})) satisfies%
\begin{equation}
\gamma_{\zeta}\geq\frac{\operatorname{Re}\zeta}{\left\vert \zeta\right\vert }.
\label{infsupfirst}%
\end{equation}
For every $F\in V^{\prime}$, problem (\ref{varformrobin}) has a
unique solution.  
In particular if there are $f\in L^{2}\left(  \Omega\right)  $, $g\in
L^{2}\left(  \Gamma\right)  $ such that $F\left(  v\right)  =\left(
f,v\right)  +\left(  g,v\right)  _{\Gamma}$, then the solution $u$ satisfies 
\begin{equation}
\left\Vert u\right\Vert _{|\zeta|}\leq\frac{1}{\operatorname{Re}\zeta}\left(
\left\Vert f\right\Vert +C\sqrt{\left\vert \zeta\right\vert }\left\Vert
g\right\Vert _{\Gamma}\right)  . \label{solopRobin}%
\end{equation}

\end{lemma}%

\proof
We follow the idea of the proof in \cite{bambduong}. We choose $v=\frac{\zeta
}{|\zeta|}u$. For the sesquilinear form with Robin boundary conditions we have%
\[
\operatorname{Re}a_{\zeta}\left(  u,\frac{\zeta}{|\zeta|}u\right)
=\frac{\operatorname{Re}\zeta}{|\zeta|}\left\Vert u\right\Vert _{|\zeta|}%
^{2}+\left\vert \zeta\right\vert \left\Vert u\right\Vert _{\Gamma}^{2}%
\geq\frac{\operatorname{Re}\zeta}{|\zeta|}\left\Vert u\right\Vert _{\left\vert
\zeta\right\vert }^{2}.
\]
The positivity of the inf-sup constant $\gamma_{\zeta}$ implies unique
solvability (see, e.g., \cite[Thm.~{2.1.44}]{SauterSchwab2010}; 
the above argument can be used to show \cite[(2.34b)]{SauterSchwab2010}).
We obtain%
\[
\left\Vert u\right\Vert _{|\zeta|}\leq\frac{\left\vert \zeta\right\vert
}{\operatorname{Re}\zeta}\sup_{v\in H^{1}\left(  \Omega\right)  \backslash
\left\{  0\right\}  }\frac{\left\vert F\left(  v\right)  \right\vert
}{\left\Vert v\right\Vert _{\left\vert \zeta\right\vert }}\leq\frac{|\zeta
|}{\operatorname{Re}\zeta}\left(  \frac{\left\Vert f\right\Vert }{|\zeta
|}+\left\Vert g\right\Vert _{L^{2}\left(  \Gamma\right)  }\sup_{v\in
H^{1}\left(  \Omega\right)  \backslash\left\{  0\right\}  }\frac{\left\Vert
v\right\Vert _{\Gamma}}{\left\Vert v\right\Vert _{|\zeta|}}\right)  .
\]
A multiplicative trace inequality in the form of (\ref{multtraceinequ}) leads
to (\ref{solopRobin}).%
\endproof

\begin{lemma}
\label{LemWellPosRobin}Let $\Omega\subset\mathbb{R}^{3}$ be a smooth domain 
that is star-shaped with respect to a ball or let $\Omega$ be 
a convex polyhedron. Let 
the functional $F\in V^{\prime}$ be of the form 
$F\left(
v\right)  =\left(  f,v\right)  +\left(  g,v\right)  _{\Gamma}$ 
with $f\in L^{2}\left(
\Omega\right)  $ and $g\in L^{2}\left(  \Gamma\right)  $. 
Then, problem
(\ref{varformrobin}) has a unique solution and satisfies
\begin{equation}
\Vert u\Vert_{|\zeta|}\leq C_{S}\left(  \frac{1}{1+\operatorname*{Re}(\zeta)}
\left\Vert f\right\Vert +\frac{1}{\sqrt{1+\operatorname*{Re}(\zeta)}%
}\left\Vert g\right\Vert _{\Gamma}\right)  \label{aprioriest}%
\end{equation}
for some $C_{S}$ independent of $\zeta\in\mathbb{C}_{\geq0}^{\circ}$.
\end{lemma}

\begin{remark}
In \cite{GrahamSpenceZou2018}, a stability estimate is proved that is related
to \eqref{aprioriest} if $\operatorname*{Re}\zeta$ is sufficiently small.
For $\zeta\in S_{\beta}^{c}$, the estimate \eqref{aprioriest} is
non-degenerate for $\operatorname*{Re}\zeta\rightarrow0$ in contrast to
(\ref{infsupfirst}) and the result in \cite{GrahamSpenceZou2018}.
\end{remark}

%

\proof
Without loss of generality, we assume that $\Omega$ is star-shaped with respect
to the origin. We will fix a parameter $\beta > 1$ sufficiently large at the 
end of the proof. We distinguish between two cases.

\textbf{Case a:} $\zeta\in S_{\beta}$\textbf{.}
The condition $|\zeta|\geq1$ leads to 
\begin{align}
\operatorname{Re}\zeta>\left(
1+\beta^{2}\right)  ^{-1/2} |\zeta|\geq \left(
1+\beta^{2}\right)  ^{-1/2} 
\end{align}
and Lemma~\ref{Lem:infsup1} becomes applicable:
\[
\gamma_{\zeta}\geq\frac{\operatorname{Re}(\zeta)}{\left\vert \zeta\right\vert
}\geq\frac{1}{\sqrt{1+\beta^{2}}},
\]
which implies \eqref{aprioriest} for $\zeta\in S_{\beta}$.

\textbf{Case b: }$\zeta\in S_{\beta}^{c}$\textbf{.} For $\operatorname{Re}%
\zeta>0$, existence and uniqueness follows from Lemma~\ref{Lem:infsup1} while
the well-posedness in the case $\operatorname{Re}\zeta=0$ is a consequence of
\cite[Prop.~{8.1.3}]{MelenkDiss}. 
We write 
$\zeta=\operatorname*{Re} \zeta + \operatorname*{i} \operatorname*{Im} \zeta=
:\nu-\operatorname*{i}k$ so
that $\zeta\in S_{\beta}^{c}$ implies $\left\vert k\right\vert \geq\beta\nu$
for $\beta>1$. First let $\nu\geq1$. We choose $v=\frac{\zeta}{|\zeta|}u$ and
consider the real part of (\ref{varformrobin}), which yields
\begin{equation}
\frac{\nu}{|\zeta|}\left\Vert \nabla u\right\Vert ^{2}+\nu|\zeta|\Vert
u\Vert^{2}+\left\vert \zeta\right\vert \left\Vert u\right\Vert _{\Gamma}%
^{2}\leq|(f,u)|+|(g,u)|. \label{ineq:testv1}%
\end{equation}
Young's inequality on the right-hand side leads to%
\[
|(f,u)|+|(g,u)|\leq\frac{1}{2\nu|\zeta|}\Vert f\Vert^{2}+\frac{\nu|\zeta|}%
{2}\Vert u\Vert^{2}+\frac{1}{2|\zeta|}\left\Vert g\right\Vert _{\Gamma}%
^{2}+\frac{|\zeta|}{2}|u|_{\Gamma}^{2}.
\]
These two inequalities imply%
\[
\Vert\nabla u\Vert^{2}+\frac{|\zeta|^{2}}{2}\Vert u\Vert^{2}+\frac{|\zeta|^{2}%
}{2\nu}\Vert u\Vert_{\Gamma}^{2}\leq\frac{1}{2\nu^{2}}\Vert f\Vert^{2}%
+\frac{1}{2\nu}\Vert g\Vert_{\Gamma}^{2},
\]
which is the desired \eqref{aprioriest} in view of $\nu \ge 1$. 

The proof for $\nu<1$ is essentially a repetition of the arguments in the
proof of \cite[Prop.~{8.1.4}]{MelenkDiss} using the inequalities for three
different test functions in \eqref{varformrobin} and Young's inequality. For
completeness, we show the relevant inequalities. 
The first test function is $v= u$ yielding, after taking the real part,
\begin{equation}
\label{eq:11a}
\|\nabla u\|^2 - (k^2 -\nu^2) \|u\|^2 + \nu \|u\|^2_\Gamma \leq 
|(f,u)| + |(g,u)_\Gamma| . 
\end{equation}
Next we choose $v=-\operatorname*{sign}(k)u$ and consider 
the imaginary part to get 
\begin{equation}
\label{eq:11b}
2|k|\nu\Vert u\Vert^{2}+|k|\Vert u\Vert_{\Gamma}^{2}\leq|(f,u)|+|(g,u)_{\Gamma
}|.
\end{equation}
As a last test function we use $v\left(  x\right)  =\left\langle x,\nabla
u\left(  x\right)  \right\rangle $; note that the assumptions on the 
domain imply via elliptic regularity theory that $v\in V$. 
Integration by parts yields
with $d = 3$ (we write $d$ to indicate the generalization to arbitrary spatial
dimension $d$) 
\begin{align*}
\operatorname{Re}a_{\zeta}\left(  u,v\right)   &  =\operatorname{Re}\left(
\left(  \nabla u,\nabla\left\langle x,\nabla u\right\rangle \right)
+\zeta^{2}\left(  u,\left\langle x,\nabla u\right\rangle \right)
+\zeta\left(  u,\left\langle x,\nabla u\right\rangle \right)  _{\Gamma}\right)
\nonumber\\
&  =\left\Vert \nabla u\right\Vert ^{2}+\frac{1}{2}\left(  x,\nabla\left(
\left\Vert \nabla u\right\Vert ^{2}\right)  \right)  +\operatorname{Re}\left(
\zeta^{2}\left(  u,\left\langle x,\nabla u\right\rangle \right)  +\zeta\left(
u,\left\langle x,\nabla u\right\rangle \right)  _{\Gamma}\right) \nonumber\\
&  =\left(1-\frac{d}{2}\right)\left\Vert \nabla u\right\Vert ^{2}+\frac{1}{2}\left(
\left\langle x,n\right\rangle ,\left\Vert \nabla u\right\Vert ^{2}\right)
_{\Gamma}+\frac{d\left(  k^{2}-\nu^{2}\right)  }{2}\left\Vert u\right\Vert
^{2}\nonumber\\
&  +\frac{\left(  \nu^{2}-k^{2}\right)  }{2}\left(  \left\langle
x,n\right\rangle u,u\right)  _{\Gamma}+\operatorname{Re}\left(  \zeta\left(
u,\left\langle x,\nabla u\right\rangle \right)  _{\Gamma}\right)  +2\nu
k\operatorname{Im}\left(  u,\left\langle x,\nabla u\right\rangle \right)
\nonumber\\
&  \leq|(f,\left\langle x,\nabla u\left(  x\right)  \right\rangle
)|+|(g,\left\langle x,\nabla u\left(  x\right)  \right\rangle)_\Gamma|.
\end{align*}
Rearranging yields 
\begin{align}
\label{lastestReazeta-a}%
&
\frac{d(k^2- \nu^2)}{2} \|u\|^2 + 
\frac{1}{2} (\langle x,n\rangle, |\nabla u|^2)_\Gamma 
\leq 
 \left(\frac{d}{2}-1\right) \|\nabla u\|^2  
+ \frac{k^2}{2} (\langle x,n\rangle |u|^2)_\Gamma \\
\nonumber 
& 
+ |\zeta| \|u\|_\Gamma \|\langle x ,\nabla u\rangle\|_\Gamma + 
2 \nu k \|u\| \|\langle x, \nabla u\rangle\| + 
| (f,\langle x,\nabla u\rangle)| + 
| (g,\langle x,\nabla u\rangle)_\Gamma|. 
\end{align}
We remark that \eqref{eq:11b} and \eqref{eq:11a} give
\begin{align}
\label{nablaucoerest-a} 
k \|u\|^2_\Gamma &\stackrel{\eqref{eq:11b}}{\leq}   
|(f,u)| + \|g\|_\Gamma \|u\|_\Gamma \leq 
|(f,u)| + \frac{1}{2k} \|g\|^2_\Gamma + \frac{k}{2} \|u\|^2_\Gamma, 
\\
\label{nablaucoerest} 
\|\nabla u\|^2 &\leq
 (k^2 -\nu^2) \|u\|^2 + |(f,u)| + |(g,u)_\Gamma|, 
\end{align}
which allows for controlling $\|u\|_\Gamma$ and 
$\|\nabla u\|$ in terms of $k \|u\|$ and the data $f$, $g$: 
\begin{align}
\label{nablaucoerest-b} 
k \|u\|^2_\Gamma &\leq 
\frac{2}{k} \|f\| (k \|u\|) + \frac{1}{k} \|g\|^2_\Gamma, \\ 
\label{nablaucoerest-c} 
\|\nabla u\|^2 &\leq
 (k^2 -\nu^2) \|u\|^2 + \frac{3}{k} \|f\| (k \|u\|) + 
\frac{2}{k} \|g\|^2_\Gamma.  
\end{align}
Since $\Omega$ is assumed to be star-shaped, one has  
$0< c_1 \leq \langle x,n(x)\rangle \leq c_2$ for all $x \in \Gamma$. 
Inserting this and 
\eqref{nablaucoerest-c} into 
\eqref{lastestReazeta-a} 
gives with $c_3 = \operatorname{diam}\Omega$ 
\begin{align*}
& (k^2- \nu^2) \|u\|^2 + 
\frac{c_1}{2} \|\nabla u\|^2_\Gamma 
\leq 
k^2 \frac{c_2}{2} \|u\|^2_\Gamma + 
\left(\frac{d}{2}-1\right) 
\left( \frac{3}{k} \|f\| k \|u\|  + \frac{2}{k} \|g\|^2_\Gamma\right)
\\
& 
+ |\zeta| \|u\|_\Gamma \|\langle x ,\nabla u\rangle\|_\Gamma + 
2 \nu k \|u\| \|\langle x, \nabla u\rangle\| + 
| (f,\langle x,\nabla u\rangle)| + 
| (g,\langle x,\nabla u\rangle)_\Gamma|. 
\end{align*}
The proof can be completed with suitable applications of Young's inequality, 
use of \eqref{nablaucoerest-b}, \eqref{nablaucoerest-c}, and 
selecting $\beta$ sufficiently large to treat the 
term $\nu k \|u\| \|\langle x,\nabla u \rangle\|
\leq c_3 \nu k \|u\| \|\nabla u\|$. 
\endproof

\subsection{The Inf-Sup Constant of $a_{\zeta}\left(  \cdot,\cdot\right)  $
for $\zeta\in S_{\beta}^{c}$}

\label{subsec:infsuphelmh}

In the following Theorem~\ref{thm:alternative-inf-sup}
we will prove an alternative estimate (compared to
(\ref{infsupfirst})) for the inf-sup constant that is robust as
$\operatorname{Re}\zeta\rightarrow0$. To estimate this constant we employ the
standard ansatz $u\in V$ and $v=u+z$ for some $z\in V$. Then%
\[
a_{\zeta}\left(  u,u+z\right)  =\left\Vert u\right\Vert _{\left\vert
\zeta\right\vert }^{2}+a_{\zeta}\left(  u,z\right)  +b_{\zeta}\left(
\gamma_{0}u,\gamma_{0}u\right)  +\left(  \zeta^{2}-\left\vert \zeta\right\vert
^{2}\right)  \left\Vert u\right\Vert ^{2}.
\]
The choice of $z$ will be related to some adjoint problem.
the next section.

\begin{theorem}
\label{thm:alternative-inf-sup}
Let $\Omega\subset\mathbb{R}^{3}$ be a smooth domain that is star-shaped with
respect to a ball or let $\Omega$ be a convex polyhedron. 
Then there exists a constant $c>0$
such that for all $\zeta\in\mathbb{C}_{\geq0}^{\circ}$ the inf-sup constant
$\gamma_f$ of (\ref{eq:inf-sup}) satisfies%
\[
\gamma_{\zeta}\geq\frac{1}{1+c\frac{\left\vert \operatorname{Im}%
\zeta\right\vert }{1+\operatorname{Re}\zeta}}.
\]

\end{theorem}

%

\proof
Let $\nu=\operatorname{Re}\zeta$ and $k=-\operatorname{Im}\zeta$ and set
$\sigma=1/\sqrt{2}$. First, we consider the case $\zeta\in\mathbb{C}_{\geq
0}^{\circ}$ with $\nu\geq\sigma$.

From Lemma~\ref{Lem:infsup1} we have for any $\zeta\in\mathbb{C}_{\geq\sigma
}^{\circ}$ the estimate%
\[
\gamma_{\zeta}\geq\frac{\operatorname{Re}\zeta}{\left\vert \zeta\right\vert
}=\frac{1}{\sqrt{1+\left(  \frac{k}{\nu}\right)  ^{2}}}\geq\frac{1}%
{1+\frac{\left\vert k\right\vert }{\nu}}\geq\frac{1}{1+c\frac{\left\vert
k\right\vert }{\nu+1}}\quad\text{for }c=1+\sqrt{2}.
\]

It remains to consider the case $\zeta\in\mathbb{C}_{\geq0}^{\circ}$ with
$\nu<\sigma$. Let $u,z\in V$ and set $v=u+z$. Then%
\begin{equation}
a_{\zeta}\left(  u,v\right)  =\left\Vert u\right\Vert _{\left\vert
\zeta\right\vert }^{2}+\left(  \zeta^{2}-\left\vert \zeta\right\vert
^{2}\right)  \left\Vert u\right\Vert ^{2}+\zeta\left(  u,u\right)  _{\Gamma
}+a_{\zeta}\left(  u,z\right)  . \label{azetareal}%
\end{equation}
We consider the adjoint problem: find $z\in V$ such that%
\begin{equation}
a_{\overline{\zeta}}\left(  z,w\right)  =\alpha^{2}\left(  u,w\right)
\quad\forall w\in V\quad\text{with\quad}\alpha^{2}:=\left\vert \zeta
\right\vert ^{2}-\overline{\zeta}^{2}=-2k\operatorname*{i}\overline{\zeta},
\label{adjproblemRobin}%
\end{equation}
which is well-posed according to Lemma~\ref{LemWellPosRobin} and satisfies%
\[
\left\Vert z\right\Vert _{|\zeta|}\leq C_{S}\left\vert \alpha\right\vert
^{2}\left\Vert u\right\Vert =2C_{S}\left\vert k\zeta\right\vert \left\Vert
u\right\Vert \leq2C_{S}\left\vert k\right\vert \left\Vert u\right\Vert
_{|\zeta|}.
\]
For this choice of $z$, we consider the real part of (\ref{azetareal}) and
obtain%
\[
\operatorname{Re}a_{\zeta}\left(  u,v\right)  \geq\left\Vert u\right\Vert
_{|\zeta|}^{2}+\nu\left\Vert u\right\Vert _{\Gamma}^{2}\geq\left\Vert
u\right\Vert _{|\zeta|}^{2}.
\]
Hence%
\[
\left\Vert v\right\Vert _{|\zeta|}\leq\left(  1+2C_{S}\left\vert k\right\vert
\right)  \left\Vert u\right\Vert _{|\zeta|}%
\]
and%
\[
\gamma_{\zeta}\geq\frac{1}{1+2C_{S}\left\vert k\right\vert }\geq\frac
{1}{1+\tilde{c}\frac{\left\vert k\right\vert }{\nu+1}}\quad\text{for }0\leq
\nu\leq\sigma.
\]%
\endproof

\section{Regular Decomposition of the Helmholtz Solution\label{SecRegDecomp}}

\label{Sec:RegDecomp} In this section, we develop a regular decomposition of
the solution of the Helmholtz problem \eqref{eq:strong-helmholtz-robin} based
on a frequency splitting of the right-hand side. The frequency splitting for
functions defined on the full space $\mathbb{R}^{3}$ is defined via their
Fourier transform (Sect.\ \ref{subsec:fullsplitting}). For functions defined
on finite domains, we derive the regular splitting using a lifting operator
(Sect. \ref{subsec:Solution-operators}). This generalizes the theory developed
in \cite{MelenkSauterMathComp, mm_stas_helm2} to complex frequencies
and the resulting estimates are explicit with respect to the real and
imaginary part of the wave number.

\subsection{The Full Space Adjoint Problem for $\zeta\in S_{\beta}^{c}$}

\label{subsec:fullsplitting}

The first result concerns the adjoint problem for the full space
$\Omega=\mathbb{R}^{3}$. Let $\phi\in L^{2}\left(  \Omega\right)  $ be a
function with compact support. We choose $R>0$ sufficiently large so that
the open ball $B_{R}$ with radius $R$ centered at the origin contains
$\operatorname*{supp}\phi$. We consider the problem
\begin{equation}%
\begin{split}
(-\Delta+\overline{\zeta^{2}})z  &  =\phi\qquad\qquad\text{ in }\mathbb{R}%
^{3},\\
\left\vert \left\langle \frac{x}{\left\Vert x\right\Vert },\nabla z\left(
x\right)  \right\rangle +\overline{\zeta}z\left(  x\right)  \right\vert  &
=o\left(  \Vert x\Vert^{-1}\right)  ~\text{ as }\Vert x\Vert\rightarrow\infty.
\end{split}
\label{adjproblem_strong}%
\end{equation}
To analyze this equation we employ Fourier transformation and introduce a cutoff
function $\mu\in C^{\infty}\left(  \mathbb{R}_{\geq0}\right)  $ satisfying :w

\begin{equation}%
\begin{array}
[c]{lll}%
\operatorname*{supp}\mu\subset\left[  0,4R\right]  , & \left.  \mu\right\vert
_{\left[  0,2R\right]  }=1, & \left\vert \mu\right\vert _{W^{1,\infty}\left(
\mathbb{R}_{\geq0}\right)  }\leq\dfrac{C_\mu}{R},\\
&  & \\
\forall x\in\mathbb{R}_{\geq0}:0\leq\mu\left(  x\right)  \leq1, & \left.
\mu\right\vert _{\left[  4R,\infty\right[  }=0, & \left\vert \mu\right\vert
_{W^{2,\infty}\left(  \mathbb{R}_{\geq0}\right)  }\leq\dfrac{C_\mu}{R^{2}}.
\end{array}
\label{mueestimates}%
\end{equation}

The fundamental solution to the Helmholtz operator $\mathcal{L}_{\zeta
}u=-\Delta u+\zeta^{2}u$ in $\mathbb{R}^{3}$ is given by%
\[
G\left(  \zeta,x\right)  :=g\left(  \zeta,\left\Vert x\right\Vert \right)
\quad\text{with\quad}g\left(  \zeta,r\right)  :=\frac{\operatorname*{e}%
^{-\zeta r}}{4\pi r}.
\]
It satisfies%
\[
\left|\left\langle \frac{x}{\left\Vert x\right\Vert },\nabla_{x}G\left(
\zeta,x\right)  \right\rangle +\zeta G\left(  \zeta,x\right)\right|  = o\left(
\left\Vert x\right\Vert ^{-1}\right)  \quad\text{for }\left\Vert x\right\Vert
\rightarrow\infty
\]
so that $z$ is given by $z = G(\overline{\zeta})\ast \phi$.
Define $M\left(  x\right)  :=\mu\left(  \left\Vert x\right\Vert \right)  $
and
\[
z_{\mu}\left(  x\right)  :=\left(  G\left(  \overline{\zeta}\right)  M\right)
\ast\phi:=\int_{B_{R}}G\left(  \overline{\zeta},x-y\right)  M\left(
x-y\right)  \phi\left(  y\right)  dy\qquad\forall x\in\mathbb{R}^{3}.
\]
The properties of $\mu$ ensure $z_{\mu}|_{B_{R}}=z|_{B_{R}}$. 
To analyze the stability and
regularity of $z_{\mu}$ we introduce a frequency splitting of the solution
$z_{\mu}=z_{H^{2}}+z_{{\mathcal{A}}}$ that depends on the complex frequency
$\zeta\in\mathbb{C}_{\geq0}$ and a parameter $\lambda\geq \lambda_0 > 1$.

\begin{lemma}
\label{lemma:decomposition} 
Let $\phi\in L^{2}\left(  \mathbb{R}^{3}\right)  $
such that $\operatorname*{supp}\phi$ is contained in a ball $B_{R}:=B_R(0)$ 
of radius $R>0$ centered at the origin, 
and let $\mu$ be a cutoff function satisfying 
(\ref{mueestimates}). Then there exists a constant $C>0$ depending only on
$R$ and $\mu$ such that 
the solution $z = G\left(\overline{\zeta}\right)\ast \phi$ 
of (\ref{adjproblem_strong}) and $z_{\mu}:=\left(  G\left(
\overline{\zeta}\right)  M\right)  \ast\phi$ satisfy $\left.  z\right\vert
_{B_{R}}=\left.  z_{\mu}\right\vert _{B_{R}}$ and%
\begin{equation}
\label{eq:lemma:decomposition-10}
\left\Vert z_{\mu}\right\Vert _{|\zeta|}\leq\frac{C}{1+\operatorname{Re}\zeta}%
\Vert\phi\Vert\qquad\forall\zeta\in\mathbb{C}_{\geq0}.
\end{equation}
Furthermore, for every $\lambda\geq \lambda_0 >1$ 
and $\zeta\in\mathbb{C}_{\geq0}$ with $\operatorname{Im}\zeta\neq 0$ 
there exists a $\lambda$- and $\zeta$-dependent splitting
$z_{\mu}=z_{H^{2}}+z_{{\mathcal{A}}}$ satisfying 
\begin{subequations}
\label{formuladecomplemma}
\end{subequations}%
\begin{align}
\left\Vert \nabla^{p}z_{H^{2}}\right\Vert  &  \leq C^\prime \frac{\lambda}{\lambda
-1}\left(  \frac{|\zeta|}{\operatorname{Im}\zeta}\right)  ^{2}\left(
\lambda\left\vert \operatorname{Im}\zeta\right\vert \right)  ^{p-2}\left\Vert
\phi\right\Vert \qquad\forall p\in\{0,1,2\},\label{formuladecomplemmaa}\\
\Vert\nabla^{p}z_{{\mathcal{A}}}\Vert &  \leq C^\prime\frac{1+\left\vert
\zeta\right\vert }{1+\operatorname{Re}\zeta}\left(  \sqrt{3}\lambda\left\vert
\operatorname{Im}\zeta\right\vert \right)  ^{p-2}\left\Vert \phi\right\Vert
\qquad\forall p\in{\mathbb{N}}_{0}. \label{formuladecomplemmab}%
\end{align}
Here,
$|\nabla^{p}z_{\mathcal{A}}|$ stands for a sum over all derivatives of order $p$
(see \eqref{eq:analytic-part}). The  constant $C'$ depends only
on $\lambda_0$, $R$, and $\mu$. 
\end{lemma}

\begin{remark}
As the estimates in 
Lemma~\ref{lemma:decomposition} degenerate for $\operatorname{Im} \zeta \rightarrow 0$, 
we will employ Lemma~\ref{lemma:decomposition} for
$\zeta \in S^c_\beta$ for fixed $\beta > 0$. 
Then $\left\vert \operatorname{Im}\zeta\right\vert \geq
\beta\operatorname{Re}\zeta$ and we have%
\begin{equation}
\left\vert \operatorname{Im}\zeta\right\vert \leq\left\vert \zeta\right\vert
\leq\tilde{C}\left\vert \operatorname{Im}\zeta\right\vert \quad\text{with\quad
}\tilde{C}:=\frac{\sqrt{1+\beta^{2}}}{\beta}. \label{equivalence}%
\end{equation}
In particular, $\zeta \in S^c_\beta$ implies $\operatorname{Im} \zeta \ne 0$. 
\end{remark}

%

\proof
For $\zeta\in\mathbb{C}_{\geq0}$, we set $\nu=\operatorname{Re}\zeta$ and
$k=-\operatorname*{Im}\zeta$. In order to construct the splitting $z=z_{H^{2}%
}+z_{\mathcal{A}}$, we start by recalling the definition of the Fourier
transformation for functions with compact support
\[
\hat{w}\left(  \xi\right)  =\mathcal{F}\left(  w\right)  =\left(  2\pi\right)
^{-d/2}\int_{\mathbb{R}^{d}}\operatorname{e}^{-\operatorname*{i}\left\langle
\xi,x\right\rangle }w\left(  x\right)  dx\qquad\forall\xi\in\mathbb{R}^{d}%
\]
and the inversion formula
\[
w\left(  x\right)  =\mathcal{F}^{-1}\left(  w\right)  =\left(  2\pi\right)
^{-d/2}\int_{\mathbb{R}^{d}}\operatorname{e}^{\operatorname*{i}\left\langle
x,\xi\right\rangle }\hat{w}\left(  \xi\right)  d\xi\qquad\forall
x\in\mathbb{R}^{d}.
\]
Next, we introduce a frequency splitting of a function $w\in L^{2}\left(
\Omega\right)  $ depending on $\zeta$ and a parameter $\lambda>1$ by using the
Fourier transformation. The low- and high-frequency part of $w$ is given by%
\begin{equation}
L_{\mathbb{R}^{3}}w:=\mathcal{F}^{-1}\left(  \chi_{\lambda\left\vert
k\right\vert }\mathcal{F}\left(  w\right)  \right)  \quad\text{and\quad
}H_{\mathbb{R}^{3}}w:=\mathcal{F}^{-1}\left(  \left(  1-\chi_{\lambda
\left\vert k\right\vert }\right)  \mathcal{F}\left(  w\right)  \right)
\label{def:LH-filter-full}%
\end{equation}
where $\chi_{\delta}$ is the characteristic function of the open ball with
radius $\delta>0$ centered at the origin.

We construct a decomposition of \(z_\mu\)
\begin{equation}
z_{\mu}=z_{H^{2}}+z_{\mathcal{A}}.
\label{eq:lemma-decomposition-100}%
\end{equation}
as follows: We decompose the
right-hand side $\phi$ in \eqref{adjproblem_strong} via%
\begin{equation}
\phi=\phi_{\left\vert k\right\vert }+\phi_{\left\vert k\right\vert }%
^{c}=L_{\mathbb{R}^{3}}\phi+H_{\mathbb{R}^{3}}\phi.
\label{eq:lemma-decomposition-f-bandlimited}%
\end{equation}
Accordingly, we define the decomposition of $z_{\mu}$ by
\begin{equation}
z_{H^{2}}:=\left(  G\left(  \overline{\zeta}\right)  M\right)  \star
\phi_{\left\vert k\right\vert }^{c}\quad\text{and\quad}z_{{\mathcal{A}}%
}:=\left(  G\left(  \overline{\zeta}\right)  M\right)  \star\phi_{\left\vert
k\right\vert }. \label{eq:lemma-decomposition-v}%
\end{equation}
%
The Fourier transform of $G\left(  \overline{\zeta},\cdot\right)  M$ is given
by%
\[
\left(  \widehat{G\left(  \overline{\zeta},\cdot\right)  M}\right)  \left(
\xi\right)  =\sigma\left(  \overline{\zeta},\left\Vert \xi\right\Vert \right)
\]
with%
\[
\sigma\left(  \zeta,s\right)  =(2\pi)^{-3/2}4\pi%
{\displaystyle\int_{0}^{\infty}}
g\left(  \zeta,r\right)  \mu\left(  r\right)  r^{2}\dfrac{\sin\left(
rs\right)  }{rs}dr.
\]
In the following we will analyze the symbol 
$\sigma\left(  \zeta,\cdot\right)  $.
We have: 
\begin{align*}
\left\vert s\sigma\left(  \zeta,s\right)  \right\vert  &  =(2\pi)^{-3/2}\left\vert
{\displaystyle\int_{0}^{\infty}}
\operatorname{e}^{-\zeta r}\mu\left(  r\right)  \sin\left(  rs\right)
dr\right\vert \\
&  \leq(2\pi)^{-3/2}\int_{0}^{4R}\operatorname*{e}\nolimits^{-\nu
r}dr=4R\sqrt{\frac{2}{\pi}}E_{0}\left(  4R\nu\right) 
\end{align*}
with \(E_{0}\left(  t\right)  :=\frac{1-\operatorname*{e}\nolimits^{-t}}{t}\leq
\frac{C_{0}}{1+t}.\)
Applying integration by parts leads to%
\begin{align*}
\sigma\left(  \zeta,s\right)   &  =(2\pi)^{-3/2}
{\displaystyle\int_{0}^{\infty}}
\operatorname{e}^{-\zeta r}\mu\left(  r\right)  \dfrac{\sin\left(  rs\right)
}{s}dr\\&=(2\pi)^{-3/2}\frac{1}{\zeta}%
{\displaystyle\int_{0}^{\infty}}
\operatorname{e}^{-\zeta r}\partial_{r}\left(  \mu\left(  r\right)
\dfrac{\sin\left(  rs\right)  }{s}\right)  dr\\
&  =(2\pi)^{-3/2}\frac{1}{\zeta}%
{\displaystyle\int_{0}^{\infty}}
\operatorname{e}^{-\zeta r}\left(  \mu^{\prime}\left(  r\right)  \dfrac
{\sin\left(  rs\right)  }{s}+\mu\left(  r\right)  \cos rs\right)  dr.
\end{align*}
This allows for the estimate
\begin{align*}
\left\vert \sigma\left(  \zeta,s\right)  \right\vert  &  = (2\pi)^{-3/2}
\frac{1}{|\zeta|}\left\vert
{\displaystyle\int_{0}^{\infty}}
\operatorname{e}^{-\zeta r}\left(  \mu^{\prime}\left(  r\right)  \dfrac
{\sin\left(  rs\right)  }{s}+\mu\left(  r\right)  \cos rs\right)
dr\right\vert \\
&  \leq(2\pi)^{-3/2}\frac{1}{|\zeta|}%
{\displaystyle\int_{0}^{4R}}
\operatorname{e}^{-\nu r}\left(  \frac{C_\mu}{R}r+1\right)  dr\\
&  \leq4R(2\pi)^{-3/2}\frac{1}{|\zeta|}\left(  4C_\mu E_{1}\left(  4\nu
R\right)  +E_{0}\left(  4R\nu\right)  \right)
\end{align*}
with%
\[
E_{1}\left(  t\right)  =\frac{1-\operatorname*{e}^{-t}(1+t)}{t^{2}}\leq
E_{0}^{2}\left(  t\right)  .
\]
Hence,%
\[
\left\vert \sigma\left(  \zeta,s\right)  \right\vert \leq 4R(2\pi)^{-3/2}
\frac{E_{0}\left(  4R\nu\right)  }{|\zeta|}\left(  1+4C_\mu E_{0}(4R\nu) \right). 
\]
Since $E_{0}\left(  t\right)  \leq1$ we end up with%
\[
\left\vert \sigma\left(  \zeta,s\right)  \right\vert \leq 4R \left(
1+4C_\mu\right)  (2\pi)^{-3/2}\frac{E_{0}\left(  4R\nu\right)  }{|\zeta|}.
\]
As a consequence, we have proved that%
\begin{align*}
|\zeta|\left\Vert z_\mu\right\Vert  &  \leq 4 R\left(  1+4C_\mu\right)
E_{0}\left(  4R\nu\right)  \left\Vert \phi\right\Vert, \\
\left\Vert \partial_{i}z_\mu\right\Vert  &  \leq 4 RE_{0}\left(  4R\nu\right)
\left\Vert \phi\right\Vert
\end{align*}
so that we have 
\[
\left\Vert z_\mu\right\Vert _{|\zeta|}\leq\sqrt{2+\left(  1+4C_\mu\right)  ^{2}%
}\left(  16\pi R\right)  E_{0}\left(  4R\nu\right)  \left\Vert \phi\right\Vert
.
\]
This shows \eqref{eq:lemma:decomposition-10}. 
In the following we estimate higher order derivatives. For the product
$s^{2}\sigma\left(  s\right)  $, we get%
\begin{align*}
\left\vert s^{2}\sigma\left(  \zeta,s\right)  \right\vert  &  =(2\pi)^{-3/2}\left\vert \int_{0}^{\infty}\operatorname{e}^{-\zeta r}\mu\left(
r\right)  s\sin\left(  rs\right)  dr\right\vert \\
&=(2\pi)^{-3/2}
\left\vert \int_{0}^{\infty}\operatorname{e}^{-\zeta r}\mu\left(  r\right)
\partial_{r}\cos\left(  rs\right)  dr\right\vert \\
&  \leq(2\pi)^{-3/2} \left(  \left\vert \int_{0}^{\infty}\cos\left(
rs\right)  \partial_{r}\left(  \operatorname{e}^{-\zeta r}\mu\left(  r\right)
\right)  dr\right\vert +1\right) \\
&  \leq(2\pi)^{-3/2}|\zeta|\left\vert \int_{0}^{\infty}\cos\left(
rs\right)  \operatorname{e}^{-\zeta r}\mu\left(  r\right)  dr\right\vert \\&\quad
+(2\pi)^{-3/2}\left(  \left\vert \int_{0}^{\infty}\cos\left(
rs\right)  \operatorname{e}^{-\nu r}\mu^{\prime}\left(  r\right)
dr\right\vert +1\right) \\
&  =:T^{\operatorname*{I}}+T^{\operatorname*{II}}.
\end{align*}
The estimates
\begin{align}
\label{eq:TI}
T^{\operatorname*{I}}  &  \leq (2\pi)^{-3/2} 4RE_{0}\left(  4R\nu\right)
|\zeta|,\\
\label{eq:TII}
T^{\operatorname*{II}}  &  \leq (2\pi)^{-3/2} 4CE_{0}\left(
4R\nu\right)
\end{align}
follow from the properties of $\mu$ (cf.~(\ref{mueestimates})). As a simple
consequence we obtain for $m\geq2$%
\begin{equation}
\label{eq:s2sigma}
\left\vert s^{2}\sigma\left(  \zeta,s\right)  \right\vert \leq (2\pi)^{-3/2} 4 \left(  C+R|\zeta|\right)E_{0}\left(  4R\nu\right)
\end{equation}
and
\begin{equation}
\label{eq:sigma-analytic}
\sup_{0<s<\lambda\left\vert k\right\vert }\left\vert s^{m}\sigma\left(
\zeta,s\right)  \right\vert \leq (2\pi)^{-3/2} 4 \left(
C+R|\zeta|\right)E_{0}\left(  4R\nu\right)  \left(  \lambda\left\vert k\right\vert
\right)  ^{m-2}.
\end{equation}
Hence for $\alpha\in\mathbb{N}_{0}^{3}$, $\left\vert \alpha\right\vert =2$, 
we have 
\[
\left\Vert \partial^{\alpha}z_\mu\right\Vert \leq 4\left(  R\left\vert
\zeta\right\vert +C\right)  E_{0}\left(  4R\nu\right)  \left\Vert
\phi\right\Vert
\]
and%
\begin{align}
\nonumber 
\Vert\nabla^{p}z_{{\mathcal{A}}}\Vert &  =\sqrt{\sum_{\substack{\alpha
\in\mathbb{N}_{0}^{3}\\\left\vert \alpha\right\vert =p}}\binom{p}{\alpha
}\left\Vert \partial^{\alpha}z_{\mathcal{A}}\right\Vert ^{2}}\leq C^{\prime
}E_{0}\left(  4R\nu\right)  \left(  1+\left\vert \zeta\right\vert \right)
\left(  \lambda\left\vert k\right\vert \right)  ^{p-2}3^{p/2}\left\Vert
\phi\right\Vert \\
\label{eq:analytic-part}
&  \leq C^{\prime\prime}\frac{1+|\zeta|}{1+\nu}\left(  \sqrt{3}\lambda
\left\vert k\right\vert \right)  ^{p-2}\left\Vert \phi\right\Vert
\qquad\forall p\in{\mathbb{N}}_{\geq2}.
\end{align}
The bounds \eqref{eq:analytic-part} expresses the desired estimate 
\eqref{formuladecomplemmab}. A direct application of \eqref{eq:s2sigma}
does not lead to \eqref{formuladecomplemmaa} as it 
introduces an undesired factor $|\zeta|$. This is removed by noting 
that is suffices to consider $s = \|\xi\|$ with $s \ge \lambda |k|$
and that only the estimates for $T^I$ need to be refined. This is 
achieved with an integration by parts: 
\begin{align}
\left\vert T^{\operatorname*{I}}\right\vert  &  =(2\pi)^{-3/2}
|\zeta|\left\vert \int_{0}^{4R}\cos\left(  rs\right)  \operatorname{e}^{-\zeta
r}\mu\left(  r\right)  dr\right\vert \nonumber\\
&  =(2\pi)^{-3/2}|\zeta|\left\vert \left(  \frac{\zeta}{\zeta
^{2}+s^{2}}+\int_{0}^{4R}\frac{\operatorname*{e}^{-\zeta r}(\zeta
\cos(rs)-s\sin(rs))}{\zeta^{2}+s^{2}}\mu^{\prime}\left(  r\right)  dr\right)
\right\vert \nonumber\\
&  \leq(2\pi)^{-3/2}\left(  \frac{|\zeta|^{2}}{\left\vert \zeta
^{2}+s^{2}\right\vert }\left(  1+\frac{C}{R}\int_{0}^{4R}\operatorname*{e}%
\nolimits^{-\nu r}dr\right)  \right. \nonumber\\
&  \phantom{\leq\sqrt{\frac{1}{2\pi}} \Big(}+\left.  \frac{\left\vert
\zeta\right\vert s}{\left\vert \zeta^{2}+s^{2}\right\vert }\left\vert \int
_{0}^{4R}\operatorname*{e}\nolimits^{-\zeta r}\sin\left(  rs\right)
\mu^{\prime}\left(  r\right)  dr\right\vert \right)  . \label{TI}%
\end{align}
Observe%
\[
\frac{|\zeta|^{2}}{\left\vert \zeta^{2}+s^{2}\right\vert }=\frac{|\zeta|^{2}%
}{\sqrt{\left(  \nu^{2}+s^{2}-k^{2}\right)  ^{2}+4\nu^{2}k^{2}}}\leq
\frac{|\zeta|^{2}}{s^{2}-k^{2}}\leq\left(  \frac{|\zeta|}{\operatorname{Im}%
\zeta}\right)  ^{2}\frac{1}{\lambda^{2}-1}.
\]
Also we have%
\[
\frac{s|\zeta|}{\nu^{2}+\left(  s^{2}-k^{2}\right)  }\leq\frac{\lambda
\left\vert k\right\vert |\zeta|}{\nu^{2}+k^{2}\left(  \lambda^{2}-1\right)
}\leq\frac{\lambda}{\lambda^{2}-1}\frac{|\zeta|}{\left\vert \operatorname{Im}%
\zeta\right\vert }.
\]
Hence,
\begin{equation}
\left\vert T^{\operatorname*{I}}\right\vert \leq (2\pi)^{-3/2} \frac{C}{\lambda-1}\left(
\frac{|\zeta|}{\operatorname{Im}\zeta}\right)  ^{2}.
\end{equation}
This leads to%
\[
\left\vert s^{2}\sigma\left(  \zeta,s\right)  \right\vert \leq (2\pi)^{-3/2}C\frac{\lambda
}{\lambda-1}\left(  \frac{|\zeta|}{\operatorname{Im}\zeta}\right)  ^{2}%
\quad\text{for }|s|\geq\lambda\left\vert k\right\vert
\]
and, in turn,%
\[
\left\vert s^{p}\sigma\left(  \zeta,s\right)  \right\vert \leq (2\pi)^{-3/2}C\frac{\lambda
}{\lambda-1}\left(  \frac{|\zeta|}{\operatorname{Im}\zeta}\right)  ^{2}\left(
\lambda\left\vert \operatorname{Im}\zeta\right\vert \right)  ^{p-2}%
\quad\text{for }|s|\geq\lambda\left\vert k\right\vert \text{, }p=0,1,2.
\]
From this, assertion (\ref{formuladecomplemmaa}) follows.%
\endproof

\subsection{The Helmholtz Solution with Robin Boundary Conditions}

In this section, we will derive a regularity result in the spirit of 
Lemma~\ref{lemma:decomposition} for $\zeta\in S_{\beta}^{c}$ for the interior
problem with Robin boundary conditions:%
\begin{equation}
-\Delta u+\zeta^{2}u=f\quad\text{in }\Omega,\qquad\partial_{n}u+\zeta
u=g\quad\text{on }\Gamma. \label{eq:robin-bc-smooth-domain}%
\end{equation}

Note that Assumption~\ref{ASmoothDomain} implies well-posedness of
(\ref{eq:robin-bc-smooth-domain}) via Lemma~\ref{LemWellPosRobin}. The 
solution operator for \eqref{eq:robin-bc-smooth-domain} is denoted 
$S_{\zeta}:L^{2}\left(
\Omega\right)  \times H^{1/2}\left(  \Gamma\right)  \rightarrow V$. 

\begin{theorem}
\label{thm:decomposition-bounded-domain}Let Assumption~\ref{ASmoothDomain} be
valid and fix \(\beta>0\).
Then there exist constants $C$, $\gamma>0$ such that for every $f\in
L^{2}(\Omega)$, $g\in H^{1/2}\left(  \Gamma\right)  $, and $\zeta\in
S_\beta^c$, 
the solution $u=S_{\zeta}(f,g)$ of
(\ref{eq:robin-bc-smooth-domain}) can be written as $u=u_{\mathcal{A}%
}+u_{H^{2}}$, where, for all $p\in\mathbb{N}_{0}$,
\begin{subequations}
\label{EstTheodebd0}
\end{subequations}%
\begin{align}
&  \Vert u_{\mathcal{A}}\Vert_{|\zeta|}\leq C\left(  \frac{1}%
{1+\operatorname*{Re}(\zeta)}\Vert f\Vert+\frac{1}{\sqrt{1+\operatorname*{Re}%
(\zeta)}}\frac{1}{\sqrt{|\zeta|}} \Gw{g}
\right)  ,\label{EstTheodebd0a}\\
&  \Vert\nabla^{p+2}u_{\mathcal{A}}\Vert_{L^{2}(\Omega)}\leq C\frac{\gamma
^{p}}{|\zeta|}\max\{p,|\zeta|\}^{p+2}\left(  \frac{1}{1+\operatorname*{Re}%
(\zeta)}\Vert f\Vert+\frac{1}{\sqrt{|\zeta|}} \Gw{g}
\right)  ,\label{EstTheodebd0b}\\
&  \Vert u_{H^{2}}\Vert_{H^{2}(\Omega)}+|\zeta|\Vert u_{H^{2}}\Vert_{|\zeta
|}\leq C\left(  \Vert f\Vert+\Vert g\Vert_{\Gamma,|\zeta|} \right)  .
\label{EstTheodebd0c}%
\end{align}

\end{theorem}

The proof is the generalization of the proof in \cite{mm_stas_helm2} for real
wave numbers to more general $\zeta\in\mathbb{C}_{\geq0}^{\circ}$ with emphasis
on the explicit dependence of the estimates on the real and imaginary part. It
follows from Lemmata~\ref{lemma:domain-contraction} and
\ref{lemma:boundary-contraction}, which are presented 
in Sect.~\ref{subsec:Solution-operators} ahead.
\endproof

\subsection{The Solution Operators $N_{\zeta}$, $S_{\zeta}^{\Delta}$,
$S_{\zeta}^{L}$, and $S^{\zeta}$\label{subsec:Solution-operators}}

For the analysis we introduce low- and high pass frequency filters for a
bounded domain as well as for its boundary. Let $E_{\Omega}:L^{2}%
(\Omega)\rightarrow L^{2}(\mathbb{R}^{3})$ be the extension operator of Stein,
\cite[Chap.~VI]{emstein}. Then for $f\in L^{2}\left(  \Omega\right)  $ we set
\begin{equation}
L_{\Omega}f:=\left.  \left(  L_{\mathbb{R}^{d}}\left(  E_{\Omega}f\right)
\right)  \right\vert _{\Omega}\quad\text{and\quad}H_{\Omega}f:=\left.  \left(
H_{\mathbb{R}^{d}}\left(  E_{\Omega}f\right)  \right)  \right\vert _{\Omega},
\label{seteq:frequency-splitting-domainomega}%
\end{equation}
for $L_{\mathbb{R}^{d}}$ and $H_{\mathbb{R}^{d}}$ defined in
(\ref{def:LH-filter-full}) for some $\lambda>1$. By 
\cite[Lemmas~{4.2}, {4.3}]{mm_stas_helm2}, these operators have the following
stability properties:
\begin{align} 
\label{eq:stability-LOmega}
\|L_\Omega f\|_{H^{s}(\Omega)} &\leq C_{s} \|f\|_{H^s(\Omega)},
\qquad s \ge 0,\\
\|H_\Omega f\|_{H^{s'}(\Omega)} &\leq C_{s,s'} 
\label{eq:stability-HOmega}
|\lambda \operatorname{Im} \zeta|^{s'-s}\|f\|_{H^s(\Omega)},
\qquad 0 \leq s' \leq s, 
\end{align} 
where the constant $C_{s}$ depends on $s$ and 
$C_{s,s'}$ depends on $s$, $s'$ but is independent of $\lambda$ and $\zeta$.

To define frequency filters
on the boundary we employ a lifting operator $G^{N}$ defined in 
Lemma~\ref{lemma:lifting} below with the mapping property
$G^{N}:H^{s}(\Gamma)\rightarrow H^{3/2+s}(\Omega)$ for every $s>0$ and
$\partial_{n}G^{N}g=g$.  We then define $H_{\Gamma}^{N}$ and $L_{\Gamma}^{N}$ by
\begin{equation}
H_{\Gamma}^{N}(g):=\partial_{n}H_{\Omega}\left(  G^{N}\left(  g\right)
\right)  ,\qquad L_{\Gamma}^{N}\left(  g\right)  :=\partial_{n}L_{\Omega
}\left(  G^{N}\left(  g\right)  \right)  .
\label{seteq:frequency-splitting-domainomegad}%
\end{equation}
In particular, we have $H_{\Gamma}^{N}:H^{1/2}\left(  \Gamma\right)
\rightarrow H^{1/2}(\Gamma)$ and $L_{\Gamma}^{N}:H^{1/2}\left(  \Gamma\right)
\rightarrow H^{1/2}(\Gamma)$.
\begin{lemma}[Def.\ of lifting $G^N$]
\label{lemma:lifting}
Let $\partial\Omega$ be smooth. 
Given $\zeta \in \mathbb{C}_{\geq0}$, 
define $u:=G^N g$ as the solution of 
$$
-\Delta u + \left|\zeta\right|^2 u = 0 \quad \mbox{ in $\Omega$}, 
\qquad \partial_n u = g. 
$$
Then the following holds: 
\begin{align}
\|G^N g\|_{\left|\zeta\right|} &\lesssim \frac{1}{\sqrt{\left|\zeta\right|}}\|g\|_{\Gamma}, \label{posdefenergyestimate} \\
\|G^N g\|_{H^2(\Omega)} &\lesssim \Gw{g}
.\label{H2Gnest}
\end{align}
\end{lemma}
\begin{proof}
The energy estimate \eqref{posdefenergyestimate} follows from the coercivity of the pertinent sesquilinear form. 
The $H^2$-estimate follows from elliptic regularity theory.
%
\end{proof}
\begin{lemma}[properties of $L_\Gamma$ and $H_\Gamma$] 
\label{lemma:LGamma-HGamma}
Let $\partial \Omega$ be smooth. 
Fix $q \in (0,1)$. Then there is $\lambda > 1$ in the definition of 
\(L_\Gamma^N\) and $H_\Gamma^N$ such that the following holds (with implied constants independent of $q$): 
\begin{align}
\label{eq:lemma:LGamma-HGamma-10}
\|L_\Gamma^N g\|_{H^{s}(\Gamma)} &\lesssim \left|\zeta\right|^{s-1/2} \|g\|_{\Gamma,|\zeta|},\qquad\qquad\quad s \in \{0,1/2\},\\
\label{eq:lemma:LGamma-HGamma-20}
\|H_\Gamma^N g\|_{H^s(\Gamma)} &\lesssim q^{1/2-s} |\zeta|^{s-1/2} \|g\|_{\Gamma,|\zeta|}, 
\quad s \in \{0,1/2\}. 
\end{align}
\end{lemma}
\begin{proof}
Recall that $L_{\Gamma}^{N}g:=\gamma_{0}g^{N}$, where 
\begin{equation}
\label{eq:gN}
g^{N}:=\langle n^{\ast},\nabla L_{\Omega}G^{N}g\rangle
\end{equation}
and $n^{\ast}$ denotes an analytic extension of the
normal $n:\Gamma\rightarrow\mathbb{S}_{2}$ on $\Omega$ to a tubular
neighborhood $T\subset\Omega$ of $\Gamma$ and $\gamma_{0}$ is the standard
trace operator. Using \eqref{multtraceinequ} yields%
\begin{align*}
\Vert L_{\Gamma}^{N}g\Vert_{\Gamma}  &  \leq C\frac{1}{\sqrt{|\zeta|}}\Vert
g^{N}\Vert_{|\zeta|,T}\\
&  =C\left(  \sqrt{|\zeta|}\Vert g^{N}\Vert_{L^{2}\left(  T\right)  }+\frac
{1}{\sqrt{|\zeta|}}\Vert\nabla g^{N}\Vert_{L^{2}\left(  T\right)  }\right) \\
&  \leq C\left(  \sqrt{|\zeta|}\Vert\nabla L_{\Omega}G^{N}g\Vert+\frac
{1}{\sqrt{|\zeta|}}\left\Vert \nabla\nabla^{\intercal}L_{\Omega}%
G^{N}g\right\Vert \right)  ,
\end{align*}
where $\nabla\nabla^{\intercal}$ denotes the Hessian of a function. From
(\ref{eq:stability-LOmega})
\begin{align*}
\Vert L_{\Gamma}^{N}g\Vert_{\Gamma}  &  \lesssim  \sqrt{|\zeta|}\Vert
G^{N}g\Vert_{H^{1}\left(  \Omega\right)  }+\frac{1}{\sqrt{|\zeta|}}\left\Vert
G^{N}g\right\Vert _{H^{2}\left(  \Omega\right)  }
\overset{\text{Lemma~\ref{lemma:lifting}}}{\lesssim
}{|\zeta|}^{-1/2} \|g\|_{\Gamma,|\zeta|}. 
\end{align*}
For \(s=1/2\), we note
\begin{align*}
\|L_\Gamma^N g\|_{H^{1/2}(\Gamma)} &\lesssim \|G^N g\|_{H^2(\Omega)}  
\lesssim \|g\|_{\Gamma,|\zeta|}. 
\end{align*}
The proof of \eqref{eq:lemma:LGamma-HGamma-20} is similar. We note 
\begin{align*}
\|H_\Omega G^N \|_{H^2(\Omega)} & 
\stackrel{(\ref{eq:stability-HOmega})}{\lesssim} \|G^N\|_{H^2(\Omega)} 
\lesssim \|g\|_{\Gamma,|\zeta|}, \\
\|H_\Omega G^N \|_{H^1(\Omega)} & 
\stackrel{(\ref{eq:stability-HOmega})}{\lesssim} q \left|\zeta\right|^{-1} \|G^N\|_{H^2(\Omega)} 
\lesssim q \left|\zeta\right|^{-1} \|g\|_{\Gamma,|\zeta|}, 
\end{align*}
where $q$ is related to $\lambda$ via (\ref{eq:stability-HOmega}) and 
can be made arbitrarily small by selecting $\lambda$ appropriately. 
Hence, recalling that $H_\Gamma^N g = \partial_n H_\Omega G^N g$ we get 
\begin{align*}
\|H_\Gamma^N g\|_{H^{1/2}(\Gamma)} &\lesssim \|G^N g\|_{H^2(\Omega)} \lesssim \|g\|_{\Gamma,|\zeta|}, \\
\|H_\Gamma^N g\|_{\Gamma} &\lesssim \|G^N g\|^{1/2}_{H^1(\Omega)}\|G^N g\|^{1/2}_{H^2(\Omega)} 
\lesssim 
q^{1/2} |\zeta|^{-1/2} \|g\|_{\Gamma,|\zeta|}. 
\end{align*}
\end{proof}

Next, we introduce the solution operators $N_{\zeta}$, $S_{\zeta}^{\Delta}$,
$S_{\zeta}^{L}$. 

\begin{enumerate}
\item We denote by $u:=N_{\zeta}f = G(\zeta) \ast f$ the solution of the full space Helmholtz
problem with Sommerfeld radiation condition (in the weak sense):
\begin{align*}
(-\Delta+\zeta^{2})u  &  =f\text{ in }\mathbb{R}^{3},\\
\left\vert \frac{\partial u}{\partial r}+\zeta u\right\vert  &  =o\left(
\Vert x\Vert^{-1}\right)  \text{ as }\Vert x\Vert\rightarrow\infty,
\end{align*}
for $f\in L^{2}(\mathbb{R}^{3})$ with compact support. Here $\partial/\partial
r$ denotes the derivative in radial direction $x/\Vert x\Vert$.

\item $S_{\zeta}^{\Delta}(g)$ is the solution operator to the problem
\[%
\begin{split}
-\Delta u+|\zeta|^{2}u  &  =0\text{ in }\Omega,\\
\partial_{n}u+\zeta u  &  =g\text{ on }\Gamma,
\end{split}
\]
for $g\in L^{2}(\Gamma).$

\item We define $S_{\zeta}^{L}(f,g):=S_{\zeta}(L_{\Omega}f,L_{\Gamma}^{N}g)$
as the solution operator to the problem \eqref{eq:strong-helmholtz-robin} for
analytic right-hand sides $L_\Omega f$, $L_\Gamma^N g$. 
\end{enumerate}

The proof of the next lemma is a direct consequence of 
Lemma~\ref{lemma:decomposition}.

\begin{lemma}
[properties of $N_{\zeta}$]\label{lemma:properties-of-Nk} Let 
$\operatorname*{Im}\zeta\neq 0$. For $f\in L^{2}(\mathbb{R}^{3})$ with
$\operatorname{supp}f\subset B_{R}:=B_R(0)$, 
the function $u=N_{\zeta}f$ satisfies
$-\Delta u+\zeta^{2}u=f$ on $B_{R}$. For any $\lambda>1$ (appearing in the definition of the operator
$H_{\mathbb{R}^{3}}$ defined in \eqref{def:LH-filter-full}) there exist 
$C>0$ depending only on $R$ and $\mu$ such that
\begin{subequations}\label{eq:Newton-stability}
\begin{align}
\Vert N_{\zeta}(H_{\mathbb{R}^{3}}f)\Vert_{|\zeta|,B_{R}}  & 
 \leq C \frac{1}{\lambda-1}\left(\frac{\left|\zeta\right|}{\left|\operatorname{Im}\zeta\right|}\right)^3
|\operatorname{Im}\zeta|^{-1}\Vert f\Vert_{L^{2}(\mathbb{R}^{3})},
\\
\Vert N_{\zeta}(H_{\mathbb{R}^{3}}f)\Vert_{H^{2}(B_{R})}  &  \leq C \frac{\lambda}{1-\lambda} \left(\frac{\left|\zeta\right|}{\operatorname{Im}\zeta}\right)^2 \Vert f\Vert_{L^{2}(\mathbb{R}^{3})}.
\end{align}
\end{subequations}
Furthermore, for $\beta> 0$ the following is true: given 
$q \in (0,1)$ one can select $\lambda > 1$ such that for all 
$\zeta\in S_{\beta}^{c}$
\begin{subequations}
 \label{eq:Newton-stability-contraction}
\begin{align}
\Vert N_{\zeta}(H_{\mathbb{R}^{3}}f)\Vert_{|\zeta|,B_{R}}  &  \leq q
|\operatorname{Im}\zeta|^{-1}\Vert f\Vert_{L^{2}(\mathbb{R}^{3}
)},
\label{eq:Newton-stability-contractiona}\\
\Vert N_{\zeta}(H_{\mathbb{R}^{3}}f)\Vert_{H^{2}(B_{R})}  &  \leq C_{\lambda,\beta} \Vert f\Vert_{L^{2}(\mathbb{R}^{3})}.
 \label{eq:Newton-stability-contractionb}
\end{align}
\end{subequations}
\end{lemma}
\begin{proof}
(\ref{eq:Newton-stability}) is a direct consequence of 
Lemma~\ref{lemma:decomposition}. The bounds 
(\ref{eq:Newton-stability-contraction}) follow from 
(\ref{eq:Newton-stability}).  
\end{proof}

The next two lemmata generalize the
results in \cite[Lemmas~{4.5}, {4.6}]{mm_stas_helm2} to complex 
wave numbers $\zeta$.   

\begin{lemma}
[properties of $S_{\zeta}^{\Delta}$]\label{lemma:properties-of-Skp}Let
$\Omega$ be a bounded Lipschitz domain and $\beta > 0$. For $g\in
L^2(\Gamma)$ the function $u=S_{\zeta}^{\Delta}(g)$
satisfies 
\begin{subequations}
\label{LaplaceDelta-Robin-problem}
\end{subequations}
\begin{subequations}
\begin{align}
\Vert u\Vert_{|\zeta|}  &  \lesssim\Vert g\Vert_{H^{-1/2}(\Gamma)},\tag{
\ref{LaplaceDelta-Robin-problem}a}\label{LaplaceDelta-Robin-problema}\\
\Vert u\Vert_{|\zeta|}  &  \lesssim|\zeta|^{-1/2}\Vert g\Vert_{\Gamma},\tag{
\ref{LaplaceDelta-Robin-problem}b}\label{LaplaceDelta-Robin-problemb}\\
\Vert u\Vert_{\Gamma}  &  \lesssim|\zeta|^{-1}\Vert g\Vert_{\Gamma} \tag{
\ref{LaplaceDelta-Robin-problem}c}\label{LaplaceDelta-Robin-problemc}
\end{align}
uniformly for all $\zeta\in S_{\beta}^{c}$.
\end{subequations}
If \(\Gamma\) is smooth and \(g\in H^{1/2}(\Gamma)\) then additionally 
\begin{align*}
\|u\|_{H^2(\Omega)} & \lesssim \Gw{g}
\end{align*}
\end{lemma}
\begin{proof}
The proof is essentially given in \cite[Lemma~{4.5}]{mm_stas_helm2}.
\end{proof}

A combination of Lemma~\ref{lemma:LGamma-HGamma} and Lemma \ref{lemma:properties-of-Snk} imply the following corollary.
\begin{corollary}
[properties of $S_{\zeta}^{\Delta}\circ H_{\Gamma}^{N}$]%
\label{lemma:properties-of-Snk}Let Assumption \ref{ASmoothDomain} be satisfied,
$\beta > 0$, and let $q\in(0,1)$. There exists $\lambda>1$ defining 
the high frequency filter $H_{\Gamma}^{N}$ such that for every 
$g\in H^{1/2}(\Gamma)$ and every $\zeta\in S_{\beta}^{c}$ we have%
\begin{align*}
\Vert S_{\zeta}^{\Delta}(H_{\Gamma}^{N}g)\Vert_{|\zeta|}  &  \leq q
 \frac{1}{|\zeta|} \Gw{g}
,\\
\Vert S_{\zeta}^{\Delta}(H_{\Gamma}^{N}g)\Vert_{H^{2}\left(  \Omega\right)  }
&  \lesssim \Gw{g}
.
\end{align*}
\end{corollary}
\begin{lemma}
[analyticity of $S_{\zeta}^{L}$]%
\label{lemma:bounded-smooth-domain-analyticity} Let
Assumption~\ref{ASmoothDomain} be valid and let $\lambda>1$ appearing in the
definition of $L_{\Omega}$ and $L_{\Gamma}^{N}$ be fixed. Then there exist
constants $C$, $\gamma>0$ independent of 
$\zeta\in\mathbb{C}_{\geq0}^{\circ}$ 
 such that,
for every $g\in H^{1/2}\left(  \Gamma\right)  $ and $f\in L^{2}\left(
\Omega\right)  ,$ the function $u_{\mathcal{A}}=S_{\zeta}(L_{\Omega
}f,L_{\Gamma}^{N}g)$ is analytic on $\Omega$ and satisfies for all
$p\in\mathbb{N}_{0}$ the estimates
\begin{align}
\Vert u_{\mathcal{A}}\Vert_{|\zeta|}  &  \leq C\left(  \frac{1}%
{1+\operatorname*{Re}(\zeta)} 
\left\Vert f\right\Vert +\frac{1}{\sqrt
{1+\operatorname*{Re}(\zeta)}}\frac{1}{\sqrt{|\zeta|}}\Gw{g}
 \right) ,\label{analytic_regularity_lf_part-a}\\
\left\Vert \nabla^{p+2}u_{\mathcal{A}}\right\Vert  &  \leq C\gamma^{p}%
\max\left\{  |\zeta|,p+2\right\}  ^{p+2}|\zeta|^{-1}\nonumber\\
&  \times\left(  \frac{1}{1+\operatorname*{Re}(\zeta)}
\left\Vert f\right\Vert+\frac{1}{\sqrt{1+\operatorname*{Re}(\zeta)}} \frac{1}{\sqrt{|\zeta|}}\Gw{g}
  \right)  . \label{analytic_regularity_lf_part}%
\end{align}

\end{lemma}

%

\proof
%
%
From Lemma~\ref{LemWellPosRobin}, we have
\begin{equation}
\Vert u_{\mathcal{A}}\Vert_{|\zeta|}\leq C\left(  \frac{1}%
{1+\operatorname*{Re}(\zeta)}\left\Vert L_{\Omega}f\right\Vert +\frac{1}%
{\sqrt{1+\operatorname*{Re}(\zeta)}}\Vert L_{\Gamma}^{N}g\Vert_{\Gamma
}\right)  . \label{uacop}%
\end{equation}
The combination of \eqref{uacop}, Lemma~\ref{lemma:lifting}, 
Lemma~\ref{lemma:LGamma-HGamma} and 
(\ref{eq:stability-LOmega}) leads to 
\[
\Vert u_{\mathcal{A}}\Vert_{|\zeta|}\leq C\left(  \frac{1}%
{1+\operatorname*{Re}(\zeta)}\left\Vert f\right\Vert +\frac{1}{\sqrt
{1+\operatorname*{Re}(\zeta)}}|\zeta|^{-1/2} \|g\|_{\Gamma,|\zeta|}\right).
\]

To estimate higher derivatives, we employ \cite[Prop.~{5.4.5}]{MelenkHabil} in a
similar way as in the proof of \cite[Lemma~{4.13}]{mm_stas_helm2}. To apply
\cite[Prop.~{5.4.5}]{MelenkHabil} an estimate of the constant%
\[
C_{G_{1}}:=\left\vert \zeta\right\vert ^{-1}\sqrt{\left\Vert g^{N}\right\Vert
_{L^{2}\left(  T\right)  }^{2}+\left\vert \zeta\right\vert ^{-2}\left\Vert
\nabla g^{N}\right\Vert _{L^{2}\left(  T\right)  }^{2}}%
\]
is needed, where $g^N$ is defined in (\ref{eq:gN}). 
 We track the dependence of $C_{G_{1}}$ on $\left\vert
\zeta\right\vert $ in a modified way (compared to \cite[p.~{1225}]%
{mm_stas_helm2}): we use inequalities (\ref{posdefenergyestimate}) and
(\ref{H2Gnest}) to obtain%
\begin{equation}
C_{G_{1}}\leq C |\zeta|^{-2} \|g\|_{\Gamma,|\zeta|}. 
 \label{eq:CG1}%
\end{equation}
Estimate (\ref{analytic_regularity_lf_part}) then follows 
from \cite[Prop.~{5.4.5}]{MelenkHabil}.
\endproof

\begin{corollary}
\label{Cor:regularity_lf_specialcase}
Fix $\beta > 0$. 
Let $f$, $\tilde{f}\in L^{2}(\Omega)$ and
$\zeta \in S_{\beta}^{c}$. Set 
$\tilde{u}=N_{\zeta}(H_{\Omega}\tilde{f})$. If $g$ has the form $g=\left(
\partial_{n}\tilde{u}+\zeta\tilde{u}\right)  $ then 
the function $u_{\mathcal{A}}=S_{\zeta}(L_{\Omega}f,L_{\Gamma
}^{N}g)$ satisfies for all $p \in {\mathbb N}_0$
\[%
\begin{split}
\left\Vert \nabla^{p+2}u_{\mathcal{A}}\right\Vert  &  \leq C_\beta\gamma^{p}%
\max\left\{  |\zeta|,p+2\right\}  ^{p+2}|\zeta|^{-1}\\
&  \times\left(  \frac{1}{1+\operatorname*{Re}(\zeta)}\left\Vert f\right\Vert
+\frac{1}{\sqrt{1+\operatorname*{Re}(\zeta)}}\frac{1}{\sqrt{|\zeta|}}%
\Vert\tilde{f}\Vert\right). 
\end{split}
\]
If $\tilde{f}=f$, this gives
\[
\left\Vert \nabla^{p+2}u_{\mathcal{A}}\right\Vert \leq C_{\beta}\gamma^{p}%
\max\left\{  |\zeta|,p+2\right\}  ^{p+2} \left\vert \zeta\right\vert
^{-1}\frac{1}{\left(  1+\operatorname*{Re}\zeta\right)  }\left\Vert
f\right\Vert .
\]

\end{corollary}

%

\proof
We proceed in the same way as in \cite[Lemma~{4.12}]{mm_stas_helm2} with
$k=\operatorname{Im}\zeta$ and estimate the constant $C_{G_{1}}$ in
\eqref{eq:CG1}. Lemma~\ref{lemma:properties-of-Nk} and (\ref{multtraceinequ})
lead to%
\begin{subequations}
\label{estofrobinterms}
\end{subequations}%
\begin{align}
\Vert\tilde{u}\Vert_{\Gamma}  &  \leq C\left\vert \zeta\right\vert
^{-1/2}\left\Vert \tilde{u}\right\Vert _{|\zeta|}\leq C\left\vert
\zeta\right\vert ^{-3/2}\Vert\tilde{f}\Vert\tag{%
\ref{estofrobinterms}%
a}\label{estofrobintermsa}\\
\Vert\tilde{u}\Vert_{H^{1/2}(\Gamma)}  &  \leq C\Vert \tilde
{u}\Vert _{H^{1}(\Omega)}\leq C\left\vert \zeta\right\vert
^{-1}\Vert \tilde{f}\Vert ,\tag{%
\ref{estofrobinterms}%
b}\label{estofrobintermsb}\\
\Vert \partial_{n}\tilde{u}\Vert _{\Gamma}  &  \leq C\left\Vert
\nabla\tilde{u}\right\Vert ^{1/2}\left\Vert \tilde{u}\right\Vert
_{H^{2}\left(  \Omega\right)  }^{1/2}\leq C\left\vert \zeta\right\vert
^{-1/2}\Vert \tilde{f}\Vert ,\tag{%
\ref{estofrobinterms}%
c}\label{estofrobintermsc}\\
\left\Vert \partial_{n}\tilde{u}\right\Vert _{H^{1/2}\left(  \Gamma\right)  }
&  \leq C\left\Vert \tilde{u}\right\Vert _{H^{2}\left(  \Omega\right)  }\leq
C\Vert \tilde{f}\Vert. \tag{%
\ref{estofrobinterms}%
d}\label{estofrobintermsd}%
\end{align}
This implies
\begin{equation}
\Vert\partial_{n}\tilde{u}+\zeta\tilde{u}\Vert_{L^{2}(\Gamma)}\lesssim\frac
{1}{\sqrt{|\zeta|}}\Vert\tilde{f}\Vert,\qquad\Vert\partial_{n}\tilde{u}%
+\zeta\tilde{u}\Vert_{H^{1/2}(\Gamma)}\lesssim\Vert \tilde{f}\Vert
, \label{eq:lemma:domain-contraction-10}%
\end{equation}
and%
\[
C_{G_{1}}:=\frac{1}{|\zeta|^{3/2}}\Vert\left(  \partial_{n}\tilde{u}%
+\zeta\tilde{u}\right)  \Vert_{\Gamma}+\frac{1}{|\zeta|^{2}}\Vert\left(
\partial_{n}\tilde{u}+\zeta\tilde{u}\right)  \Vert_{H^{1/2}(\Gamma)}\leq
C\left\vert \zeta\right\vert ^{-2}\Vert\tilde{f}\Vert.
\]
In the same way as at the end of the proof of Lemma
\ref{lemma:bounded-smooth-domain-analyticity} we obtain%
\[%
\begin{split}
\left\Vert \nabla^{p+2}u_{\mathcal{A}}\right\Vert  &  \leq C_{\beta}\gamma
^{p}\max\left\{  |\zeta|,p+2\right\}  ^{p+2}|\zeta|^{-1}\\
&  \times\left(  \frac{1}{1+\operatorname*{Re}(\zeta)}\left\Vert f\right\Vert
+\frac{1}{\sqrt{1+\operatorname*{Re}(\zeta)}}\frac{1}{\sqrt{|\zeta|}}%
\Vert\tilde{f}\Vert+\frac{1}{|\zeta|}\Vert\tilde{f}\Vert\right)  .
\end{split}
\]
\endproof

\begin{lemma}
[properties of $S_{\zeta}(f,0)$]\label{lemma:domain-contraction} Let
$\beta > 0$, Assumption~\ref{ASmoothDomain} be valid, 
and $\zeta\in S_{\beta}^{c}$. For every $q\in(0,1)$, there exist
constants $C$, $K>0$, depending on \(\beta\) such that for every $f\in L^{2}%
(\Omega)$ and $\zeta\in S_{\beta}^{c}$, the function $u=S_{\zeta}(f,0)$ can be
written as $u=u_{\mathcal{A}}+u_{H^{2}}+\widetilde{u}$, where
\begin{align*}
\Vert u_{\mathcal{A}}\Vert_{|\zeta|}  &  \leq\frac{C}%
{1+\operatorname*{Re}(\zeta)}\Vert f\Vert,\\
\Vert\nabla^{p+2}u_{\mathcal{A}}\Vert &  \leq\frac{C}%
{1+\operatorname*{Re}(\zeta)}|\zeta|^{-1}K^{p}\max\{p+2,|\zeta|\}^{p+2}\Vert
f\Vert\qquad\forall p\in\mathbb{N}_{0},\\
\Vert u_{H^{2}}\Vert_{|\zeta|}  &  \leq q|\zeta|^{-1}\Vert f\Vert,\\
\Vert u_{H^{2}}\Vert_{H^{2}(\Omega)}  &  \leq C\Vert f\Vert.
\end{align*}
For a function $\widetilde f$ with $\|\widetilde f\| \leq q \|f\|$ 
the remainder $\widetilde{u}=S_{\zeta}(\widetilde{f},0)$ satisfies
\[
-\Delta\widetilde{u}+\zeta^{2}\widetilde{u}=\widetilde{f},\qquad\partial
_{n}\widetilde{u}+\zeta\widetilde{u}=0.
\]
\end{lemma}%

\proof
Define
\[
u_{\mathcal{A}}^{\operatorname*{I}}:=S_{\zeta}(L_{\Omega}f,0),\qquad u_{H^{2}%
}^{\operatorname*{I}}:=N_{\zeta}(H_{\Omega}f).
\]
Here, the parameter $\lambda$ defining the filter operators $L_{\Omega}$ and
$H_{\Omega}$ is still at our disposal and will be selected at the end of the
proof. Then, $u_{\mathcal{A}}^{\operatorname*{I}}$ satisfies the desired
bounds by Lemma~\ref{lemma:bounded-smooth-domain-analyticity}.
Lemma~\ref{lemma:properties-of-Nk} gives
\[
\Vert u_{H^{2}}^{\operatorname*{I}}\Vert_{|\zeta|}\leq q^{\prime}|\zeta
|^{-1}\Vert f\Vert\qquad\mbox{ and }\qquad\Vert u_{H^{2}}^{\operatorname*{I}%
}\Vert_{H^{2}(\Omega)}\leq C\Vert f\Vert.
\]
Also, the parameter $q^{\prime}\in(0,1)$ depends on $\lambda$ and is still at
our disposal. In fact, in view of the statement of
Lemma~\ref{lemma:properties-of-Nk} it can be made sufficiently small by taking
$\lambda$ sufficiently large.

The function $u^{\operatorname*{I}}:=u-(u_{\mathcal{A}}^{\operatorname*{I}%
}+u_{H^{2}}^{\operatorname*{I}})$ solves
\begin{equation}
-\Delta u^{\operatorname*{I}}+\zeta^{2}u^{\operatorname*{I}}=0,\qquad
\partial_{n}u^{\operatorname*{I}}+\zeta u^{\operatorname*{I}}=-\left(
\partial_{n}u_{H^{2}}^{\operatorname*{I}}+\zeta u_{H^{2}}^{\operatorname*{I}%
}\right)  . \label{HEforu1}%
\end{equation}

Next, we define the functions $u_{\mathcal{A}}^{\operatorname*{II}}$ and
$u_{H^{2}}^{\operatorname*{II}}$ by
\[
u_{\mathcal{A}}^{\operatorname*{II}}:=S_{\zeta}\left(  0,-L_{\Gamma}%
^{N}\left(  \partial_{n}u_{H^{2}}^{\operatorname*{I}}+\zeta u_{H^{2}%
}^{\operatorname*{I}}\right)  \right)  ,\qquad u_{H^{2}}^{\operatorname*{II}%
}:=S_{\zeta}^{\Delta}\left(  -H_{\Gamma}^{N}\left(  \partial_{n}u_{H^{2}%
}^{\operatorname*{I}}+\zeta u_{H^{2}}^{\operatorname*{I}}\right)  \right)  .
\]
Then, the analytic part $u_{\mathcal{A}}^{\operatorname*{II}}$ satisfies again
the desired analyticity bounds by
Lemma~\ref{lemma:bounded-smooth-domain-analyticity} and
Corollary~\ref{Cor:regularity_lf_specialcase} . For the function $u_{H^{2}%
}^{\operatorname*{II}}$ we obtain from Lemma~\ref{lemma:properties-of-Snk} and
inequalities \eqref{estofrobinterms} (set $\tilde{u}=u_{H^{2}}%
^{\operatorname*{I}}$) the estimates%
\begin{align*}
\Vert u_{H^{2}}^{\operatorname*{II}}\Vert_{|\zeta|}  
&\leq q^{\prime}|\zeta|^{-1} \Gw{\partial_{n}u_{H^{2}}^{\operatorname*{I}}+\zeta u_{H^{2} }^{\operatorname*{I}}}
\leq C{q^{\prime}}|\zeta
|^{-1}\Vert f\Vert,\\
\Vert u_{H^{2}}^{\operatorname*{II}}\Vert_{H^{2}(\Omega)} &\lesssim \Gw{\partial_{n}u_{H^{2}}^{\operatorname*{I}}+\zeta u_{H^{2} }^{\operatorname*{I}}}
\lesssim \Vert f\Vert.
\end{align*}
Let $\nu=\operatorname{Re}\zeta$ and $k=-\operatorname{Im}\zeta$. We now set
$u_{\mathcal{A}}:=u_{\mathcal{A}}^{\operatorname*{I}}+u_{\mathcal{A}%
}^{\operatorname*{II}}$ and $u_{H^{2}}:=u_{H^{2}}^{\operatorname*{I}}%
+u_{H^{2}}^{\operatorname*{II}}$ and conclude that the function $\widetilde
{u}:=u-(u_{\mathcal{A}}+u_{H^{2}})$ satisfies%
\[
-\Delta\widetilde{u}+\zeta^{2}\widetilde{u}=\widetilde{f}:=2\left(
k^{2}+\operatorname*{i}\nu k\right)  u_{H^{2}}^{\operatorname*{II}}%
,\qquad\partial_{n}\widetilde{u}+\zeta\widetilde{u}=0.
\]
For $\widetilde{f}$ we obtain
\[
\Vert\widetilde{f}\Vert\leq C|\zeta|\Vert u_{H^{2}}^{\operatorname*{II}}%
\Vert_{|\zeta|}\leq C{q^{\prime}}\Vert f\Vert.
\]
Hence, by taking $\lambda$ sufficiently large so that $q^{\prime}$ is
sufficiently small, we arrive at the desired bound.%
\endproof

\begin{lemma}
[properties of $S_{\zeta}(0,g)$]\label{lemma:boundary-contraction} Let
$\beta>0$ and Assumption~\ref{ASmoothDomain} be valid. Let $q\in(0,1)$. 
Then there exist
constants $C$, $K>0$ independent of $\zeta\in S_{\beta}^{c}$ (but depending on \(\beta\)) such that
for every $g\in H^{1/2}(\Gamma)$ the function $u=S_{\zeta}(0,g)$ can be
written as $u=u_{\mathcal{A}}+u_{H^{2}}+\widetilde{u}$, where 
for all $p \in {\mathbb N}_0$
\begin{align*}
\Vert u_{\mathcal{A}}\Vert_{|\zeta|}  &  \leq\frac{C }{\sqrt
{1+\operatorname*{Re}\zeta}} \frac{1}{\sqrt{|\zeta|}}\Gw{g}
,\\
\Vert\nabla^{p+2}u_{\mathcal{A}}\Vert &  \leq C |\zeta|^{-1}K
^{p}\max\{p+2,|\zeta|\}^{p+2}\frac{1}{\sqrt{1+\operatorname{Re}(\zeta)}} \frac{1}{\sqrt{|\zeta|}}\Gw{g}
,\\
\Vert u_{H^{2}}\Vert_{|\zeta|}  &  \leq q \frac{1}{|\zeta|}\Gw{g}
,\\
\Vert u_{H^{2}}\Vert_{H^{2}(\Omega)}  &  \leq C \Gw{g}
.
\end{align*}
For a $\widetilde g$ with $\|\widetilde g\|_{\Gamma,|\zeta|} 
\leq \|g\|_{\Gamma,|\zeta|}$ the remainder 
$\widetilde{u}=S_{\zeta}(0,\widetilde{g})$ satisfies 
the equation%
\[
-\Delta\widetilde{u}+\zeta^{2}\widetilde{u}=0\qquad\partial_{n}\widetilde
{u}+\zeta\widetilde{u}=\widetilde{g}.
\]

\end{lemma}%

\proof
The proof is very similar to that of Lemma~\ref{lemma:domain-contraction}.
Define
\[
u_{\mathcal{A}}^{\operatorname*{I}}:=S_{\zeta}(0,L_{\Gamma}^{N}g)\qquad
\mbox{ and }\qquad u_{H^{2}}^{\operatorname*{I}}:=S_{\zeta}^{\Delta}%
(H_{\Gamma}^{N}g).
\]
Then $u_{\mathcal{A}}^{\operatorname*{I}}$ is analytic and satisfies the
desired analyticity estimates by
Lemma~\ref{lemma:bounded-smooth-domain-analyticity}. For $u_{H^{2}%
}^{\operatorname*{I}}$ we have by Corollary~\ref{lemma:properties-of-Snk}
\begin{align}
\Vert u_{H^{2}}^{\operatorname*{I}}\Vert_{|\zeta|}  &  \leq q^{\prime} \frac{1}{|\zeta|}\Gw{g}
,\label{lemma:boundary-contraction-20}\\
\Vert u_{H^{2}}^{\operatorname*{I}}\Vert_{H^{2}(\Omega)}  &  \lesssim \Gw{g}
\end{align}
where $q^{\prime}\in\left(  0,1\right)  $ is at our disposal and depends on
the parameter $\lambda$ in the definition of $H_{\Gamma}^{N}$ and $L_{\Gamma
}^{N}$. Upon abbreviating $\nu=\operatorname{Re}\zeta$ 
and $k=-\operatorname{Im}\zeta$
the function $u^{\operatorname*{I}}:=u_{\mathcal{A}}^{\operatorname*{I}%
}+u_{H^{2}}^{\operatorname*{I}}$ satisfies
\[
-\Delta u^{\operatorname*{I}}+\zeta u^{\operatorname*{I}}=-2\underbrace
{\left(  k^{2}+\operatorname*{i}\nu k\right)  }_{=\operatorname*{i}k \zeta
}u_{H^{2}}^{\operatorname*{I}},\qquad\partial_{n}u^{\operatorname*{I}}+\zeta
u^{\operatorname*{I}}=g
\]
together with
\begin{equation}
\Vert2\operatorname*{i}k \zeta u_{H^{2}}^{\operatorname*{I}}\Vert\leq
C|\zeta|\Vert u_{H^{2}}^{\operatorname*{I}}\Vert_{|\zeta|}
\stackrel{(\ref{lemma:boundary-contraction-20})}{\leq} Cq^{\prime
}\Vert g\Vert_{\Gamma,|\zeta|}. \label{lemma:boundary-contraction-10}%
\end{equation}
Next, we define $u_{\mathcal{A}}^{\operatorname*{II}}$ and $u_{H^{2}%
}^{\operatorname*{II}}$ by
\[
u_{\mathcal{A}}^{\operatorname*{II}}:=S_{\zeta}\left(  L_{\Omega}\left(
2\left(  k^{2}+\operatorname*{i}\nu k\right)  u_{H^{2}}^{\operatorname*{I}%
}\right)  ,0\right)  \mbox{ and }u_{H^{2}}^{\operatorname*{II}}:=N_{\zeta
}\left(  H_{\Omega}\left(  2\left(  k^{2}+\operatorname*{i}\nu k\right)
u_{H^{2}}^{\operatorname*{I}}\right)  \right)  .
\]
Here, in order to apply the operator $N_{\zeta}$, we extend $H_{\Omega}\left(
2\left(  k^{2}+\operatorname*{i}\nu k\right)  u_{H^{2}}^{\operatorname*{I}%
}\right)  $ by zero outside of $\Omega$. By
Lemma~\ref{lemma:bounded-smooth-domain-analyticity} and
(\ref{lemma:boundary-contraction-10}), we see that $u_{\mathcal{A}%
}^{\operatorname*{II}}$ satisfies the desired analyticity estimates. For the
function $u_{H^{2}}^{\operatorname*{II}}$, we obtain from
Lemma~\ref{lemma:properties-of-Nk}
\begin{align*}
\Vert u_{H^{2}}^{\operatorname*{II}}\Vert_{|\zeta|}  &  \leq q^{\prime}%
|\zeta|^{-1}\Vert 2(k^{2}+\operatorname{i} \nu k) u_{H^{2}}^{\operatorname*{I}}\Vert\leq Cq^{\prime}\Vert
u_{H^{2}}^{\operatorname*{I}}\Vert_{|\zeta|}
\stackrel{(\ref{lemma:boundary-contraction-20})}{\leq} C\left(  {q^{\prime}}\right)
^{2}|\zeta|^{-1}\Vert g\Vert_{\Gamma,|\zeta|},\\
\Vert u_{H^{2}}^{\operatorname*{II}}\Vert_{H^{2}(\Omega)}  &  \leq
C\Vert|\zeta|^{2}u_{H^{2}}^{\operatorname*{I}}\Vert\lesssim |\zeta|\Vert
u_{H^{2}}^{\operatorname*{I}}\Vert_{|\zeta|}
\stackrel{(\ref{lemma:boundary-contraction-20})}{\lesssim} q'\Vert
g\Vert_{\Gamma,|\zeta|}.
\end{align*}
We set $u_{\mathcal{A}}:=u_{\mathcal{A}}^{\operatorname*{I}}+u_{\mathcal{A}%
}^{\operatorname*{II}}$ and $u_{H^{2}}:=u_{H^{2}}^{\operatorname*{I}}%
+u_{H^{2}}^{\operatorname*{II}}$. Then $u_{\mathcal{A}}$ and $u_{H^{2}}$
satisfy the desired estimates and $\widetilde{u}:=u-(u_{\mathcal{A}}+u_{H^{2}%
})$ satisfies
\[
-\Delta\widetilde{u}+\zeta^{2}\widetilde{u}=0,\qquad\partial_{n}\widetilde
{u}+\zeta\widetilde{u}=\widetilde{g}:=-\left(  \partial_{n}u_{H^{2}%
}^{\operatorname*{II}}+\zeta u_{H^{2}}^{\operatorname*{II}}\right)
\]
with%
\begin{align*}
\Vert\widetilde{g}\Vert_{\Gamma,|\zeta|}  
&  \lesssim  
|\zeta|^{3/2} \|u_{H^{2}}\|_{\Gamma} + 
|\zeta|^{1/2} \|\partial_n u_{H^{2}}\|_{\Gamma} + 
|\zeta|\Vert u_{H^{2}%
}^{\operatorname*{II}}\Vert_{H^{1/2}(\Gamma)}+\left\Vert \partial_{n}u_{H^{2}%
}^{\operatorname*{II}}\right\Vert _{H^{1/2}\left(  \Gamma\right)  } \\
&  \leq C^\prime\left(  |\zeta|\Vert u_{H^{2}}^{\operatorname*{II}}\Vert_{|\zeta
|}+\left\Vert u_{H^{2}}^{\operatorname*{II}}\right\Vert _{H^{2}\left(
\Omega\right)  }\right)  \leq C^{\prime\prime}q^{\prime}\Vert g\Vert_{\Gamma,|\zeta|}.
\end{align*}
The result follows by selecting $\lambda$ sufficiently large so that
$q^{\prime}$ is sufficiently small.%
\endproof

\section{Discretization\label{sec:discretization}}

We apply the regularity theory of the previous section to the of $hp$-finite
element method. Let $\tilde{S}_{\zeta}$ be the solution operator of the
adjoint problem: find $z\in V$ such that%
\begin{equation}
a_{\overline{\zeta}}\left(  z,w\right)  =\left(  u,w\right)  \quad\forall w\in
V. \label{PDE:adjoint}%
\end{equation}
Let $S\subset V$ be a closed subspace and define the adjoint approximability
\[
\eta(S):=\sup_{f\in L^{2}(\Omega)\backslash\left\{  0\right\}  }\inf_{v\in
S}\frac{\Vert\tilde{S}_{\zeta}f-v\Vert_{|\zeta|}}{\Vert f\Vert}.
\]

\subsection{Discrete Inf-Sup Constant $\gamma_{\operatorname*{disc}}$ and
Quasi-Optimality\label{subsec:quasiopt}}

For $\operatorname{Re}\zeta>0$, the existence and uniqueness of the Galerkin
solution follows from Lemma~\ref{Lem:infsup1}. If $\zeta=-\operatorname{i}k$
is purely imaginary, well-posedness and quasi-optimality of the Galerkin
discretization are shown in \cite{mm_stas_helm2} under the restriction that
\[
\left\vert k\right\vert \eta(S)\leq\frac{1}{4(1+C_{b})},
\]
where $C_{b}$ is the constant appearing in \eqref{eq:contRobin}. In the next
theorem, we derive an estimate of the discrete inf-sup constant for general
$\zeta\in\mathbb{C}_{\geq0}^{\circ}$.

\begin{theorem}
\label{Th:discrinfsup} For $\zeta\in\mathbb{C}_{\geq0}$ 
let the sesquilinear form $a_\zeta$ be given by \eqref{varformrobin}.  
Then the discrete inf-sup constant
\[
\gamma_{\operatorname*{disc}}:=\inf_{u\in S\backslash\left\{  0\right\}  }%
\sup_{v\in S\backslash\left\{  0\right\}  }\frac{\left\vert a_{\zeta}\left(
u,v\right)  \right\vert }{\left\Vert u\right\Vert _{\left\vert \zeta
\right\vert }\left\Vert v\right\Vert _{|\zeta|}}%
\]
satisfies the following: 
\begin{enumerate}
\item If $\operatorname*{Re}\zeta>0$, then
\[
\gamma_{\operatorname{disc}}\geq\frac{\operatorname{Re}\zeta}{|\zeta|}.
\]

\item If $\zeta\in\mathbb{C}_{\geq0}^{\circ}$ and $\frac{(\operatorname*{Im}%
\zeta)^{2}}{|\zeta|}\eta(S)\leq\frac{1}{4(1+C_{b})}$ then,%
\begin{equation}
\gamma_{\operatorname{disc}}\geq c\frac{1+\operatorname{Re}\zeta}{\left\vert
\zeta\right\vert }, \label{gammdisc2}%
\end{equation}
for a constant $c$ independent of $\zeta$.
\end{enumerate}
\end{theorem}

\begin{remark}
The resolution condition (\ref{gammdisc2}) is not an artifact of the theory:
in \cite[Ex.~{3.7}]{MPS13}, a domain $\Omega$, a finite element space $S$, and a
purely imaginary wave number $\zeta=-\operatorname*{i}k$ are presented where the
Galerkin discretization leads to a system matrix that is not invertible.
\end{remark}

\textbf{Proof of Theorem~\ref{Th:discrinfsup}. }Let $\zeta=\nu
-\operatorname*{i}k$. The first statement follows directly from the continuous
inf-sup constant in Lemma~\ref{Lem:infsup1}. We prove the second statement.
Let $u\in S$ and choose $v=u+z$, where $z=2k^{2}\tilde{S}_{\zeta}\left(
u\right)  $. Then it is simple to check that
\[
\operatorname{Re}a(u,u+z)\geq\Vert u\Vert_{|\zeta|}^{2}.
\]
Let $z_{S}\in V$ be the best approximation of $z$ with respect to the
$\left\Vert \cdot\right\Vert _{\left\vert \zeta\right\vert }$ norm. Then
\begin{align*}
\operatorname{Re}a(u,u+z_{S})  &  =\operatorname{Re}a(u,u+z)+\operatorname{Re}%
a(u,z_{S}-z)\\
&  \geq\Vert u\Vert_{|\zeta|}^{2}-(1+C_{b})\Vert u\Vert_{|\zeta|}\Vert
z-z_{S}\Vert\\
&  \geq\Vert u\Vert_{|\zeta|}^{2}-2k^{2}(1+C_{b})\eta(S)\Vert u\Vert_{|\zeta
|}\Vert u\Vert\\
&  \geq\left(  1-2\frac{k^{2}}{|\zeta|}(1+C_{b})\eta(S)\right)  \Vert
u\Vert_{|\zeta|}^{2}\\
&  \geq\frac{1}{2}\Vert u\Vert_{|\zeta|}^{2}.
\end{align*}
Moreover%
\begin{align*}
\left\Vert u+z_{S}\right\Vert _{\left\vert \zeta\right\vert }  &  \leq\Vert
u\Vert_{\left\vert \zeta\right\vert }+\Vert z-z_{S}\Vert_{\left\vert
\zeta\right\vert }+\Vert z\Vert_{|\zeta|}\\
&  \leq\left(  1+\frac{1}{2(1+C_{b})}+C_{S}\frac{2k^{2}}{\left(  1+\nu\right)
\left\vert \zeta\right\vert }\right)  \left\Vert u\right\Vert _{\left\vert
\zeta\right\vert }%
\end{align*}
and, in turn, we have proved%
\begin{equation}
\gamma_{\text{$\operatorname*{disc}$}}\geq\frac{\operatorname{Re}a(u,u+z_{S}%
)}{\Vert u\Vert_{|\zeta|}\Vert u+z_{S}\Vert_{|\zeta|}}\geq\frac{2}{2+\frac
{1}{1+C_{b}}+\frac{4k}{|\zeta|}\frac{k}{\nu+1}C_{S}}. \label{gammdiscfinal}%
\end{equation}
A simple calculation shows that there exists a constant $c>0$ independent of
$\zeta\in\mathbb{C}_{\geq0}^{\circ}$ such that the right-hand side in
(\ref{gammdiscfinal}) is bounded from below by the right-hand side in
(\ref{gammdisc2}).%
\endproof

\begin{theorem}
\label{th:quasioptnupositiv} Assume that $\operatorname*{Re}\zeta>0$. Then the
Galerkin method based on $S$ is quasi-optimal, i.e., for every $u\in
V$ there exists a unique $u_{S}\in S$ with $a(u-u_{S},v)-b(u-u_{S}%
,v)=0$ for all $v\in S$, and
\begin{align}
\Vert u-u_{S}\Vert_{|\zeta|}  &  \leq\frac{|\zeta|}{\operatorname*{Re}(\zeta)
}(1+C_{b})\inf_{v\in S}\Vert u-v\Vert_{|\zeta|}.\label{eq:quasiopt3}\\
\Vert u-u_{S}\Vert_{L^{2}(\Omega)}  &  \leq(1+C_{b})\eta(S)\Vert u-u_{S}%
\Vert_{|\zeta|}. \label{eq:quasiopt4}%
\end{align}

\end{theorem}

Equation \eqref{eq:quasiopt3} is a direct consequence of the discrete inf-sup
constant proved in Theorem \ref{Th:discrinfsup}. Estimate \eqref{eq:quasiopt4}
follows from the proof of the next theorem (see (\ref{L2zetanormest})). We
note here that for $\zeta\in S_{\beta}$, the ratio $|\zeta|/\operatorname*{Re}%
\zeta$ is bounded from above and no resolution assumption is required. In the
next theorem, we find that under a resolution assumption, the estimate
\eqref{eq:quasiopt3} can be improved, such that it is non-degenerate for
$\operatorname*{Re}\zeta\longrightarrow0$.

\begin{theorem}
\label{th:quasioptresolution} If
\begin{equation}
\operatorname*{Re}\zeta\geq0\qquad\text{ and }\qquad\frac{(\operatorname*{Im}%
\zeta)^{2}}{|\zeta|}\eta(S)\leq\frac{1}{4(1+C_{b})},
\label{eq:resolutionscondquasiopt}%
\end{equation}
then the Galerkin method based on $S$ is quasi-optimal and
\begin{align}
\Vert u-u_{S}\Vert_{|\zeta|}  &  \leq2(1+C_{b})\inf_{v\in S}\Vert
u-v\Vert_{|\zeta|},\label{eq:quasiopt1}\\
\Vert u-u_{S}\Vert_{L^{2}(\Omega)}  &  \leq(1+C_{b})\eta(S)\Vert u-u_{S}%
\Vert_{|\zeta|}. \label{eq:quasiopt2}%
\end{align}

\end{theorem}

%

\proof
We prove the theorem in the case where $\nu=\operatorname{Re}\zeta\geq0$. Let
$e:=u-u_{S}$ and define $\psi:=\tilde{S}_{\zeta}e$. Let $\psi_{S}$ be the best
approximation to $\psi$ with respect to the $\left\Vert \cdot\right\Vert
_{\left\vert \zeta\right\vert }$ norm. The Galerkin orthogonality implies%
\begin{align*}
\Vert e\Vert^{2}  &  =a_{\zeta}(e,\psi)=a_{\zeta}(e,\psi-\psi_{S})\leq
(1+C_{b})\Vert e\Vert_{|\zeta|}\Vert\psi-\psi_{S}\Vert_{|\zeta|}\\
&  \leq(1+C_{b})\eta(S)\Vert e\Vert_{|\zeta|}\Vert e\Vert.
\end{align*}
This yields
\begin{equation}
\Vert e\Vert\leq(1+C_{b})\eta(S)\Vert e\Vert_{|\zeta|} \label{L2zetanormest}%
\end{equation}
in both cases. Let $k=-\operatorname*{Im}\zeta$. We compute for $v\in S$
\begin{align*}
\Vert e\Vert_{|\zeta|}^{2}  &  \geq\operatorname{Re}\left(  a_{\zeta
}(e,e)+2k^{2}\Vert e\Vert^{2}\right) \\
&  \leq\operatorname{Re}\left(  a_{\zeta}(e,u-v)+2k^{2}\Vert e\Vert^{2}\right)
\\
&  \leq(1+C_{b})\Vert e\Vert_{|\zeta|}\Vert u-v\Vert_{\zeta}+2\frac{k^{2}%
}{|\zeta|}(1+C_{b})\eta(S)\Vert e\Vert_{|\zeta|}^{2},
\end{align*}
which leads to \eqref{eq:quasiopt1} under the condition $\frac{k^{2}}{|\zeta
|}\eta(S)\leq\frac{1}{4(1+C_{b})}$.
\endproof

\subsection{Impact on $hp$-FEM Approximation}
\label{sec:hp-FEM}

We have shown in Sect.~\ref{subsec:quasiopt}, that the Galerkin solution
$u_{S} \in S$ of the Helmholtz problem with Robin boundary conditions
\eqref{eq:robin-bc-smooth-domain} with $\zeta\in S_{\beta}^{c}$ is
quasi-optimal for any closed subspace $S\subset V$, if the adjoint
approximability $\eta(S)$ fulfills the resolution condition
\[
\frac{(\operatorname*{Im}\zeta)^{2}}{|\zeta|}\eta(S)\leq\frac{1}{4(1+C_{b})}.
\]
Let $S_{hp}$ be the $hp$-FEM space described in 
\cite[Sect.~{5}]{MelenkSauterMathComp}. Similarly as 
in \cite{mm_stas_helm2, MelenkSauterMathComp}, one can show 
that the Galerkin method based on $S_{hp}$ is quasi-optimal if
\begin{equation}
\frac{|\zeta|h}{p}\leq C\quad\text{and}\quad p\geq C\log\left(
\operatorname*{e}+\frac{\left\vert \operatorname*{Im}(\zeta)\right\vert
}{1+\operatorname*{Re}(\zeta)}\right).  \label{eq:resolutionass}%
\end{equation}
More specifically, one can prove that there exist constants $C$, $\sigma>0$ that
depend on the shape regularity of the triangulation such that for ever $f\in
L^{2}(\Omega)$ the function $u=\tilde{S}_{|\zeta|}(f)=\overline{S_{|\zeta
|}(\overline{\alpha f},0)}$ satisfies for the regular decomposition 
$u=u_{\mathcal{A}}+u_{\mathcal{H}^{2}}$ given by 
Theorem~\ref{thm:decomposition-bounded-domain} 
\begin{subequations}
\label{eq:hp_err}
\begin{align}
& \frac{|\operatorname*{Im}\zeta|^{2}}{|\zeta|}\inf_{w\in S}\Vert u_{H^{2}%
}-w\Vert_{|\zeta|}    \leq C\frac{|\operatorname*{Im}\zeta|}{|\zeta|}\left(
\frac{|\operatorname*{Im}\zeta|h}{p}+\left(  \frac{|\operatorname*{Im}\zeta
|h}{p}\right)  ^{2}\right)  \Vert f\Vert,\label{eq:hp-errH2}\\
\nonumber 
& \frac{|\operatorname*{Im}\zeta|^{2}}{|\zeta|}\inf_{w\in S}\Vert u_{\mathcal{A}%
}-w\Vert_{|\zeta|}    \leq \\
& \qquad \qquad C\frac{|\operatorname*{Im} \zeta|^2}{|\zeta|} 
\frac{1}{1+\operatorname*{Re}(\zeta)}\left(
\frac{1}{p}+\frac{|\zeta|h}{\sigma p}\right)  
\left(  
\frac{h}{p}
+\left(  \frac{|\zeta|h}{\sigma
p}\right)  ^{p}\right)  \Vert f\Vert, \label{eq:hp_errA}%
\end{align}
\end{subequations}
(see \cite[Sect.~5]{MelenkSauterMathComp}, in particular the proof
of \cite[Thm.~{5.5}]{MelenkSauterMathComp} for details). 
By choosing $h$ and $p$ as in
(\ref{eq:resolutionass}) the right-hand sides in \eqref{eq:hp-errH2} and
\eqref{eq:hp_errA} imply the resolution assumption
\eqref{eq:resolutionscondquasiopt} and therefore the optimal convergence for
the Galerkin solution.

If $\zeta\in S_{\beta}$ no resolution condition is needed for the
quasi-optimality of the problem (cf.~Theorem~\ref{th:quasioptnupositiv}). In
that case, the solution is typically smooth in the domain and exhibits, for
large $\operatorname{Re}\zeta$, a boundary layer. Such problems can be handled 
by suitable meshes capable to resolve the layers such as Shishkin meshes
in the context of the $h$-version of the FEM 
\cite{miller1996fitted, Shi91, MelenkHabil} 
and ``spectral boundary layer meshes''
in the context of the $hp$-FEM, \cite{schwab-suri96,MelenkHabil}.  

\section{Numerical Experiments}
\label{sec:numerics}
We consider the domain \(\Omega = B_1(0)\subset \mathbb{R}^2\) and the equation
\begin{align*}
-\Delta u + \zeta^2 u =1 \qquad &\text{in } \Omega,\\
\partial_n u +  \zeta u = 0 \qquad &\text{on } \Gamma= \partial\Omega. 
\end{align*}
Using Bessel functions and polar coordinates, the solution is given as 
\begin{align*}
u(r) = c_1 J_0(\im \zeta r) + \zeta^{-2}, 
\qquad c_1 = \frac{\im}{\zeta^2}\,\frac{1}{J_1(\im \zeta)-\im J_0(\im\zeta)}.
\end{align*}
We consider values of \(\zeta\) with
\[\zeta = |\zeta| e^{\im\alpha} ,\]
where 
\begin{align*}
\alpha &= \frac{\pi}{2} (1-\widetilde \alpha), 
\qquad 
\widetilde\alpha \in \{0, 2^{-6}, 2^{-4}, 2^{-2}, 2^{-1}, 1\}, \\
|\zeta| &\in \{1,10,50,100\}.
\end{align*}
The purely imaginary wave number corresponds to the choice \(\alpha = \pi/2\) and \(\alpha = 0\) to the real-valued case. 
We consider the $h$-FEM on quasi-uniform meshes for $p \in \{1,2,3,4\}$. 
The results are presented Fig.~\ref{fig:convergenceplots}, where the 
{\bf error} is plotted versus the number of degrees of freedom per wavelength 
\[
N_{|\zeta|} = \frac{2\pi \sqrt{DOF}}{|\zeta|\sqrt{|\Omega|}} = \mathcal{O}\left(\frac{p}{h|\zeta|} \right).
\]
The calculations were carried out within the $hp$-FEM framework 
NgSolve, \cite{schoeberlNGSOLVE,schoeberl97}. 
The following features are visible in Fig.~\ref{fig:convergenceplots}:  
\begin{enumerate}[a)]
\item 
\label{item:a}
A plateau before convergence sets in.
\item 
\label{item:b}
A pollution effect for \(\zeta\) close to the imaginary axis (\(\alpha=\pi/2\)).
That is, asymptotic quasi-optimality sets in for larger $N_{|\zeta|}$
as $|\zeta|$ becomes larger for $\operatorname{Arg} \zeta$ close to $\pi/2$. 
\item 
\label{item:c}
The pollution effect decreases with increasing polynomial degree. 
In particular, the asymptotic behavior is reached for smaller values
of $N_{|\zeta|}$ as $p$ is increased. 
\item 
\label{item:d}
The pollution effect decreases with decreasing angle \(\alpha\).
\end{enumerate}
The observation \ref{item:a} reflects a natural resolution condition
for the problem class under consideration; that is, the best approximation error can only be
expected to be small if $N_{|\zeta|} \sim |\zeta|h/p$ is small. 
The pollution effect observed in \ref{item:b} is well-documented for 
the purely imaginary case $\operatorname{Re} \zeta = 0$. 
Fig.~\ref{fig:convergenceplots} shows that it is present also for 
$\operatorname{Re} \zeta \ne 0$ (and large $\operatorname{Im} \zeta$), 
albeit in a mitigated form. Theorem~\ref{th:quasioptresolution} quantifies 
how this pollution effect is weakened as the ratio 
$\operatorname{Re}\zeta/\operatorname{Im} \zeta$ increases. 
More specifically, 
the resolution condition (\ref{eq:resolutionass}), which results from 
applying Theorem~\ref{th:quasioptresolution} to high order methods, 
illustrates the helpful effect of $\operatorname{Re} \zeta \ne 0$. 
In the limiting case $\operatorname{Im} \zeta = 0$, the Galerkin method
is an energy projection method and even monotone convergence can be expected 
in the energy norm on sequences of nested meshes. 

The observation \ref{item:c} is also well-documented for the purely
imaginary case $\operatorname{Re} \zeta = 0$ and mathematically explained 
in \cite{MelenkSauterMathComp,mm_stas_helm2}. The regularity of the present
work permits to extend the $hp$-FEM analysis of 
\cite{MelenkSauterMathComp,mm_stas_helm2} to the case 
$\operatorname{Re} \zeta \ne 0$ as done in Sect.~\ref{sec:hp-FEM}. 
The observation that the asymptotic convergence regime is reached for
smaller $N_{|\zeta|}$ as $p$ is increased can be understood qualitatively 
from Theorem~\ref{th:quasioptresolution} and the bounds 
(\ref{eq:hp_err}) for $\eta$. Consider, for notational simplicity, 
the case $\operatorname{Re} \zeta = 0$. Then quasi-optimality of the 
$hp$-FEM is reached if 
$$
|\zeta| \eta(S) \lesssim 
\left(1 + \frac{h|\zeta|}{p}\right)
\left( \frac{h|\zeta|}{p} + 
|\zeta| \left(\frac{h|\zeta|}{\sigma p}\right)^p 
\right) 
\stackrel{!}{\lesssim} 1. 
$$
Recalling 
$N_{|\zeta|} = O(h|\zeta|/p)$ allows us to simplify the condition for 
quasi-optimality as 
$$
\frac{1}{N_{|\zeta|}} + 
|\zeta| \left(\frac{1}{\sigma N_{|\zeta|}}\right)^p \stackrel{!}{\lesssim} 1.  
$$
This shows that for larger $p$ quasi-optimality of the $hp$-FEM may be 
expected for small $N_{|\zeta|}$. 

Finally, observation~\ref{item:d} can again be explained by
Theorem~\ref{th:quasioptresolution} since the factor 
$(\operatorname{Im} \zeta)^2/|\zeta|$ is reduced as the ratio 
$\operatorname{Re} \zeta/\operatorname{Im}\zeta$ increases. 

\begin{sidewaysfigure}
\includegraphics[width=1\textwidth]{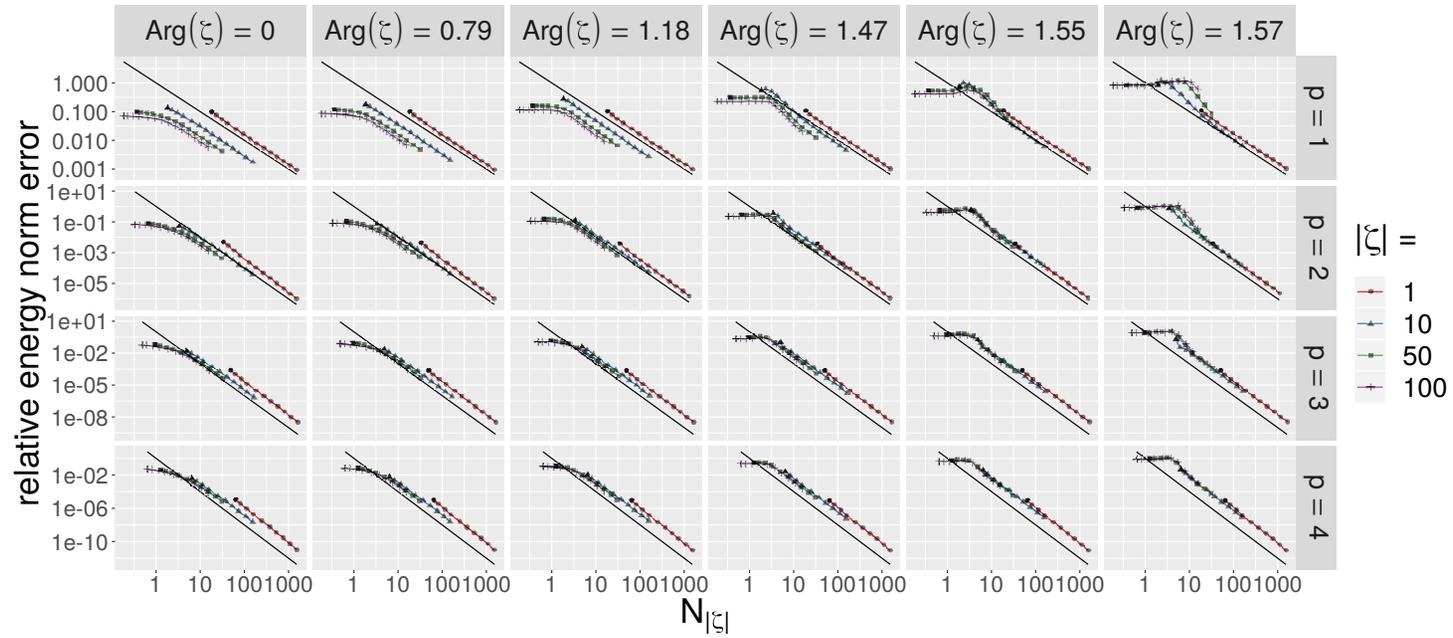}
\caption{Plots with \(H^1\)-seminorm for \(\zeta = |\zeta|\exp(\im \operatorname{Arg}\zeta)\), \(|\zeta| \in \{1,10,50,100\}\), \(\operatorname{Arg}\zeta = \frac{\pi}{2}(1-\widehat{\alpha})\) for \(\widetilde\alpha \in \{0, 2^{-6}, 2^{-4}, 2^{-2}, 2^{-1}, 1\}\), \(p \in \{1,2,3,4\}\) and different \textquotedblleft number of degree of freedom per wavelength\textquotedblright{} \(N_{|\zeta|}\). }
\label{fig:convergenceplots}
\end{sidewaysfigure}

\paragraph*{Acknowledgements:} The authors are grateful to Maximilian Bernkopf (TU Wien) for providing the numerical experiments. JMM 
acknowledges the support of the Austrian Science Fund (FWF) through project W1245. The research was initiated while 
SAS was visiting the Erwin Schr\"odinger Institute during the thematic program 
{\sl Numerical Analysis of Complex PDE Models in the Sciences}. The authors SAS and CT gratefully acknowledge the support by the Swiss National Science Foundation under grant no. 172803.
\bibliographystyle{plain}
\bibliography{complex}

\newcommand{\noopsort}[1]{} \newcommand{\printfirst}[2]{#1}
  \newcommand{\singleletter}[1]{#1} \newcommand{\switchargs}[2]{#2#1}
  \def\cprime{$'$} \def\cprime{$'$} \def\cprime{$'$}
\begin{thebibliography}{10}

\bibitem{bambduong}
A.~Bamberger and T.~Ha-{D}uong.
\newblock Formulation variationelle espace-temps pour le calcul par potentiel
  retard{\'{e}} d'une onde acoustique.
\newblock {\em Math. Meth. Appl. Sci.}, 8:405--435 and 598--608, 1986.

\bibitem{esterhazy-melenk12}
S.~Esterhazy and J.~M. Melenk.
\newblock On stability of discretizations of the {H}elmholtz equation.
\newblock In I.G. Graham, T.Y. Hou, O.~Lakkis, and R.~Scheichl, editors, {\em
  Numerical {A}nalysis of {M}ultiscale {P}roblems}, volume~83 of {\em Lect.
  Notes Comput. Sci. Eng.}, pages 285--324. Springer, Berlin, 2012.

\bibitem{GrahamSpenceZou2018}
I.~G. {Graham}, E.~A. {Spence}, and J.~{Zou}.
\newblock {Domain Decomposition with local impedance conditions for the
  Helmholtz equation}.
\newblock {\em ArXiv e-prints}, June 2018.
\newblock arXiv:1806.03731.

\bibitem{Grisvard85}
P.G. Grisvard.
\newblock {\em Elliptic {P}roblems in {N}onsmooth {D}omains}.
\newblock Pitman, Boston, 1985.

\bibitem{lopezsauter_contour}
M.~Lopez-Fernandez and S.~A. Sauter.
\newblock Fast and stable contour integration for high order divided
  differences via elliptic functions.
\newblock {\em Math. Comp.}, 84(293):1291--1315, 2015.

\bibitem{MelenkDiss}
J.~M. Melenk.
\newblock {\em On {G}eneralized {F}inite {E}lement {M}ethods}.
\newblock PhD thesis, University of Maryland at College Park, 1995.

\bibitem{MelenkHabil}
J.~M. Melenk.
\newblock {\em hp-{F}inite {E}lement {M}ethods for {S}ingular {P}erturbations}.
\newblock Springer, Berlin, 2002.

\bibitem{MPS13}
J.~M. Melenk, A.~Parsania, and S.~A. Sauter.
\newblock General {DG}-methods for highly indefinite {H}elmholtz problems.
\newblock {\em J. Sci. Comput.}, 57(3):536--581, 2013.

\bibitem{MelenkSauterMathComp}
J.~M. Melenk and S.~A. Sauter.
\newblock Convergence {A}nalysis for {F}inite {E}lement {D}iscretizations of
  the {H}elmholtz equation with {D}irichlet-to-{N}eumann boundary condition.
\newblock {\em Math. Comp}, 79:1871--1914, 2010.

\bibitem{mm_stas_helm2}
J.~M. Melenk and S.~A. Sauter.
\newblock Wave-{N}umber {E}xplicit {C}onvergence {A}nalysis for {G}alerkin
  {D}iscretizations of the {H}elmholtz {E}quation.
\newblock {\em SIAM J. Numer. Anal.}, 49(3):1210--1243, 2011.

\bibitem{miller1996fitted}
J.J.H. Miller, E.~O'Riordan, and G.I. Shishkin.
\newblock {\em Fitted Numerical Methods for Singular Perturbation Problems:
  Error Estimates in the Maximum Norm for Linear Problems in One and Two
  Dimensions}.
\newblock World Scientific, 1996.

\bibitem{SauterSchwab2010}
S.~A. Sauter and C.~Schwab.
\newblock {\em Boundary Element Methods}.
\newblock Springer, Heidelberg, 2010.

\bibitem{schoeberlNGSOLVE}
Joachim Sch\"oberl.
\newblock Finite {E}lement {S}oftware {NETGEN}/{NGS}olve version 6.2.
\newblock \url{https://ngsolve.org/}.

\bibitem{schoeberl97}
Joachim Sch{\"o}berl.
\newblock {NETGEN} - {A}n advancing front 2{D}/3{D}-mesh generator based on
  abstract rules.
\newblock {\em Computing and Visualization in Science}, 1(1):41--52, Jul 1997.

\bibitem{schwab-suri96}
Christoph Schwab and Manil Suri.
\newblock The p and hp versions of the finite element method for problems with
  boundary layers.
\newblock {\em Mathematics of Computation}, 65(216):1403--1429, 1996.

\bibitem{Shi91}
G.~I. Shishkin.
\newblock {Grid approximation of a singularly perturbed boundary-value problem
  for a quasi-linear elliptic equation in the completely degenerate case}.
\newblock {\em U.S.S.R. Comput. Math. Math. Phys.}, 31:33--46, 1991.

\bibitem{emstein}
E.~M. Stein.
\newblock {\em Singular Integrals and Differentiability Properties of
  Functions}.
\newblock Princeton, University Press, Princeton, N.J., 1970.

\bibitem{vorlander2007auralization}
Michael Vorl{\"a}nder.
\newblock {\em Auralization: fundamentals of acoustics, modelling, simulation,
  algorithms and acoustic virtual reality}.
\newblock Springer Science \& Business Media, 2007.

\end{thebibliography}

\end{document}